\documentclass[11pt,a4paper]{article}
\usepackage[utf8]{inputenc}
\usepackage{microtype}
\usepackage{amsmath,amsthm,amssymb}
\usepackage{mathrsfs} 
\usepackage[T1]{fontenc}
\usepackage{ae,aecompl}
\usepackage{times}
\usepackage[british]{babel}
\usepackage{aliascnt}
\usepackage[colorlinks,pagebackref,hypertexnames=false]{hyperref} 
\usepackage{url}
\usepackage[alphabetic,backrefs]{amsrefs}
\usepackage{nth}
\usepackage{float}
\usepackage[all]{xy}
\usepackage{tikz-cd} 
\usepackage{bookmark}
\usepackage[margin=1.9cm]{geometry}
\usepackage{enumerate}
\newcommand{\rd}{{\rm d}}

\newcommand{\ri}{{\rm i}}

\newcommand{\rG}{{\rm G}}

\newcommand{\rI}{{\rm I}}

\newcommand{\bx}{{\bf x}}

\newcommand{\bR}{{\bf R}}

\newcommand{\cL}{\mathcal{L}}

\newcommand{\cS}{\mathcal{S}}

\newcommand{\sX}{\mathscr{X}}

\newcommand{\fa}{{\mathfrak a}}

\newcommand{\fg}{{\mathfrak g}}

\newcommand{\fh}{{\mathfrak h}}

\newcommand{\fl}{{\mathfrak l}}
\newcommand{\fm}{{\mathfrak m}}

\newcommand{\R}{\mathbb{R}}
\newcommand{\C}{\mathbb{C}}

\newcommand{\su}{\mathfrak{su}}

\newcommand{\gl}{\mathfrak{\gl}}
\newcommand{\fsl}{\mathfrak{sl}}

\newcommand{\SU}{{\rm SU}}
\newcommand{\GL}{\mathrm{GL}}

\newcommand{\Curl}{\mathrm{Curl}}

\newcommand{\Ad}{\mathrm{Ad}}
\newcommand{\Aut}{\mathrm{Aut}}

\newcommand{\Diff}{\mathrm{Diff}}
\newcommand{\End}{{\mathrm{End}}}

\renewcommand{\epsilon}{\varepsilon}

\newcommand{\Ric}{{\rm Ric}}

\newcommand{\ad}{\mathrm{ad}}

\renewcommand{\Re}{\mathop{\mathrm{Re}}}

\newcommand{\tr}{\mathop{\mathrm{tr}}\nolimits}

\newcommand{\vol}{\mathrm{vol}}

\newcommand{\gt}{\texorpdfstring{\mathrm{G}_2}{\space}}

\renewcommand{\div}{\mathrm{div}}

\newcommand{\der}{\mathrm{Der}}

\newcommand{\qandq}{\quad\text{and}\quad}
\newcommand{\qwithq}{\quad\text{with}\quad}

\newcommand{\qforq}{\quad \text{for} \quad}
\newcommand{\qwhereq}{\quad \text{where} \quad}

\def\<{\mathopen{}\left<}
\def\>{\right>\mathclose{}}
\def\({\mathopen{}\left(}
\def\){\right)\mathclose{}}

\usepackage{multicol, color}

\definecolor{gold}{rgb}{0.85,.66,0}
\definecolor{cherry}{rgb}{0.9,.1,.2}
\definecolor{burgundy}{rgb}{0.8,.2,.2}
\definecolor{orangered}{rgb}{0.85,.3,0}
\definecolor{orange}{rgb}{0.85,.4,0}
\definecolor{olive}{rgb}{.45,.4,0}
\definecolor{lime}{rgb}{.6,.9,0}
\definecolor{green}{rgb}{.2,.7,0}
\definecolor{grey}{rgb}{.4,.4,.2}
\definecolor{brown}{rgb}{.4,.3,.1}

\newtheorem{theorem}{Theorem}
\newtheorem{prop}{Proposition}
\newtheorem{cor}{Corollary}

\newtheorem{lemma}[prop]{Lemma}

\numberwithin{substep}{step}

\numberwithin{subcase}{case}

\theoremstyle{remark}
\newtheorem{remark}{Remark}

\theoremstyle{definition}
\newtheorem{definition}{Definition}
\newtheorem{example}{Example}
\DeclareMathOperator{\rivarphi}{\ri_{\varphi}}
\DeclareMathOperator{\ripsi}{\ri_{\psi}}

\DeclareMathOperator{\Gl}{Gl}
\newcommand{\Addresses}{{% additional braces for segregating \footnotesize
  \bigskip
  \footnotesize

  \textsc{IMECC, University of Campinas, Campinas - São Paulo, Brazil}\par\nopagebreak
 \quad A.~Moreno: \texttt{andres.moreno@ime.unicamp.br} 

  \bigskip

  \textsc{Federal University of Cear{\'a}, Fortaleza - Cear{\'a}, Brazil}\par\nopagebreak
  \quad\quad J.~Saavedra: \texttt{julieth.p.saavedra@gmail.com}
 
}}

\usepackage{todonotes}
\usepackage{comment}

\numberwithin{equation}{section}
\numberwithin{theorem}{section}
\numberwithin{prop}{section}
\numberwithin{cor}{section}
\numberwithin{definition}{section}
\numberwithin{remark}{section}
\numberwithin{example}{section}

\allowdisplaybreaks
\usepackage{titling}
\usepackage{changepage}

\begin{document}
	\title{On the Laplacian coflow of invariant $\gt$-structures and its solitons}

 \author{Andrés J. Moreno and Julieth Saavedra}

  \date{\today}

\maketitle
	
\vspace{-0.5cm}	
\begin{abstract}
In this work, we approach the Laplacian coflow of a coclosed $\gt$-structure $\varphi$ using the formulae for the irreducible $\gt$-decomposition of the Hodge Laplacian and the Lie derivative of the Hodge dual $4$-form of $\varphi$. In terms of this decomposition, we characterise the conditions for a vector field as an infinitesimal symmetry of a coclosed $\gt$-structure, as well as the soliton condition for the Laplacian coflow. More specifically, we provide an easier proof for the absence of compact shrinking solitons of the Laplacian coflow. Moreover, we revisit  the Laplacian coflow of coclosed $\gt$-structures on almost Abelian Lie groups addressed by Fino-Bagaglini \cite{Fino2018}. However, our approach is based on  the bracket flow point of view. Notably, by showing that the norm of the Lie bracket is strictly decreasing, we prove that we have long-time existence for any coclosed Laplacian coflow solution. 
%In this work, we deal with the Laplacian coflow of a coclosed $\gt$-structure $\varphi$ using the formulae for the Hodge Laplacian and the Lie derivative of the Hodge dual $4$-form of $\varphi$, in terms of the irreducible $\gt$-decomposition for the space of $4$-forms, we characterise the conditions for a vector field be an infinitesimal symmetry of a coclosed $\gt$-structure, as well as the soliton condition for the Laplacian coflow. Especially, we provide an easier proof for the absence of compact shrinking solitons of the Laplacian coflow. Finally, we revisit  the Laplacian coflow of coclosed $\gt$-structures on almost Abelian Lie groups addressed by Fino-Bagaglini \cite{Fino2018}, however, our approach is based on  the bracket flow point of view, notably, we prove long time existence for any coclosed Laplacian coflow solution, by showing that the norm of the Lie bracket is strictly decreasing. 
\end{abstract}
	
\begin{adjustwidth}{0.95cm}{0.95cm}
    \tableofcontents
\end{adjustwidth}

%\newpage
\section{Introduction}

A $\gt$-structure is defined by a positive $3$-form $\varphi$, which, in turn, defines the metric $g$ and the corresponding Hodge dual $4$-form $\psi:=\ast\varphi$. The main goal in $\gt$-geometric is the study of \emph{torsion free} $\gt$-structures, i.e. $\nabla\varphi=0$, which is equivalent to the \emph{closed} $\rd\varphi=0$ and the \emph{coclosed} condition $\rd\psi=0$ (e.g \cite{Fernandez1982}). Using Ricci flow ideas,  Bryant introduced the \emph{Laplacian flow} of closed  $\gt$-structures \cite{Bryant2006}, which is an evolution of an initial closed $\gt$-structure along its Hodge Laplacian, namely
\begin{equation}\label{eq:Laplacian_flow}\frac{\partial\varphi(t)}{\partial t}=\Delta_{t}\varphi(t),%=dd^{\ast}\varphi 
\quad \varphi(0)=\varphi. 
\end{equation}
%is a flow of closed $3$-forms, that is $d\varphi=0$ and its well-posedness is not standard since the evolution equation is weakly parabolic
%only in the direction of closed forms. 
The Laplacian flow is not parabolic, however, when the initial condition is closed, the flow \eqref{eq:Laplacian_flow} preserves the closed condition and it evolves as a Ricci-like flow on $\Omega^3$. It allows using DeTurck's trick and, then, the Laplacian flow becomes parabolic in the direction of closed forms. In  \cite{Bryant2011}, Bryant and Xu addressed this approach in order to prove the short time existence of \eqref{eq:Laplacian_flow}.
%The short-time existence of the flow on compact manifolds was proved by Bryant and Xu \cite{Bryant2011},  introducing a gauge fixing of the flow
%called \emph{Laplacian-DeTurck flow} and then applying Nash-Moser theorem. Moreover, on a compact manifold $M$, this flow can be interpreted as the gradient flow of the Hitchin functional $V$ given by
%$$ V(\varphi)=\frac{1}{7}\int_M\varphi\wedge\psi.$$
%The functional $V$ is then the volume of the manifold $M$. It was shown by Nigel Hitchin in \cite{Hitchin2001} that if $\varphi$ is closed, then the critical points of the functional $V$ within the cohomology class $[\varphi]$ correspond precisely to torsion-free $\gt$-structures, and in particular, these critical points are maxima in the directions transverse to diffeomorphis. Under the Laplacian flow $V$ increases monotonically, so if the growth of $V$ is bounded, then $\varphi(t)$ would be expected to approach a torsion free $\gt$-structure as $t\rightarrow \infty$. The stability and analityticy of this flow has been proved by Lotay and Wey \cite{Lotay2017, Lotay2018, Lotay2019} 

Motivated by Bryant and Xu ideas on the Laplacian flow of closed $\gt$-structures, Karigiannis, McKay and Tsui introduced the \emph{Laplacian coflow}  of coclosed $\gt$-structures in \cite{Karigiannis2012}. It means that, instead of considering the heat flow equation for $\varphi$, they deal with the flow: %other similar flows have been introduced in $\gt$-geometry as the Laplacian coflow studied by Karigiannis, McKay and Tsui \cite{Karigiannis2012}. Instead of considering the heat flow equation for $\varphi$, they deal with the flow:
\begin{equation}\label{eq:Laplacian_coflow_intro}
    \frac{\partial\psi(t)}{\partial t}=\Delta_{t}\psi(t), \quad \psi(0)=\psi.
\end{equation}
Equally to the Laplacian flow, if the initial condition satisfies $\rd\psi=0$, the flow \eqref{eq:Laplacian_coflow_intro} preserves the coclosed condition. %and in fact, it preserves the cohomology class of $\psi$. 
On one side, the Laplacian coflow is interesting, because coclosed $\gt$-structures exist in any (compact and non-compact) spin and orientable $7$-manifold by a parametric $h$-principle 
(see \cite{Crowley2015a}). % Therefore, coclosed $\gt$-structures exist on any (compact and non-compact) $7$-manifold admitting a $\gt$-structure, which just requires the $7$-manifold to be oriented and spin. Although the Laplacian flow and coflow are similar from the geometric point
Unfortunately, the analytic approach employed for the Laplacian flow does not apply in the case  \eqref{eq:Laplacian_coflow_intro}, since it is not parabolic in the direction of the coclosed forms. Hence, the short-time  existence of the Laplacian coflow is still an open problem. 
%of view, it turns out that their defining equations are quite different from the analytic point of view and the well-posedness of the Laplacian coflow is still an open problem.
Nevertheless, in \cite{Grigorian2013}, Grigorian proposed a modification of \eqref{eq:Laplacian_coflow_intro} fixing the failure of the Laplacian coflow to be parabolic, specifically the \emph{modified Laplacian coflow} of coclosed $\gt$-structures is the evolution given by
\begin{equation}\label{Eq: modified Laplacian coflow}
    \frac{\partial\psi}{\partial t}=\Delta_{t}\psi(t)+2\rd((A-\tr_{g(t)} T(t))\varphi(t)), \qforq A>0.
\end{equation}
However, the critical points of \eqref{Eq: modified Laplacian coflow} are no longer torsion-free $\gt$-structures. For instance, if $\varphi$ is a \emph{nearly parallel} $\gt$-structure, i.e. $\rd\varphi=4\psi$, it vanishes the left hand side of \eqref{Eq: modified Laplacian coflow} for $A=5$. So, despite the fact that the modified Laplacian coflow can be seen as a tool for improving the torsion of $\varphi$, it does not search only for the torsion-free ones.

Regardless of the absence of an analytical theory of the Laplacian coflow in the general setting, the flow \eqref{eq:Laplacian_coflow_intro} had received the attention of some authors for manifolds with either a symmetry or an additional geometrical structure. For instance:

Assuming short time existence and uniqueness of \eqref{eq:Laplacian_coflow_intro}, in \cite{Karigiannis2012}, Karigiannis, McKay and Tsui studied soliton solutions on warped products of a circle or an interval with a compact $6$-manifold $N$ with $\SU(3)$-structure $(\omega,\Re(\Omega))$. Running the Laplacian coflow among cohomogeneity-one solutions, when $(N,\omega,\Re(\Omega))$ is a Calabi-Yau manifold, they proved that the unique soliton solutions on the warped product are the steady ones. In particular, in the compact case, the soliton solutions are given by translations and phase rotations of the standard torsion-free $\gt$-structure. 

%Using the ideas given by Lotay and Wey in \cite{Lotay2015} for the Laplacian-DeTurck flow around torsion-free $\gt$-structures, Bedulli and Vezzoni in \cite{bedulli2020stability} shown the stability of the modified Laplacian coflow when the parameter $A$ is zero around torsion-free $\gt$-structures. Note that in \cite{Grigorian2013} it is suggested to consider only the case $A>0$ and big enough, in order to ensure at least initially that the volume increases.        

On Sasakian manifolds $(M,\xi,\eta,\Phi)$ with a contact Calabi-Yau structure $(\omega:=\rd\eta,\Re(\Upsilon))$ in \cite{lotay2022CCY}, Lotay, Sá Earp
and Saavedra proved the existence of a family of $\gt$-structures by solving the Laplacian coflow, choosing $\epsilon\in \R^*$ and initial data $\varphi=\epsilon\eta\wedge\omega+\Re(\Upsilon)$, which is coclosed and the solution exists in $t\in(-\frac{1}{10\epsilon^2}, \infty)$. Hence, the solution of the Laplacian coflow is immortal with a finite singularity at $t=-\frac{1}{10\epsilon^2}$. It was the first example of compact solution to the Laplacian coflow which had an infinite time type $IIB$ singularity.

On $3$-Sasakian manifolds, there exists two non-equivalent nearly parallel $\gt$-structures  \cite{Friedrich1997}. Moreover, using the natural $\SU(2)$-action, there is a $4$-parameter family of coclosed $\gt$-structures (up to sign), which contains the nearly parallel ones. Under a special Ansatz of this family of coclosed $\gt$-structures, Kennon and Lotay proved that any solution of the Laplacian coflow starting at a coclosed $\gt$-structure converges, after rescaling, to one of the nearly parallel $\gt$-structure in the same family of the initial data  \cite{Lotay2022}. In particular, the nearly parallel $\gt$-structures are both stable within their families.

%Moreover, on $3$-Sasakian manifolds \cite{Lotay2022}, Kennon and Lotay proved that the Laplacian coflow solution, given by a two disjoint $3$-parameter families of coclosed starting at  any coclosed $\gt$-structure in either of those families converges, after rescaling, to the dual of the nearly parallel $\gt$-structure in the corresponding family. In particular, the nearly parallel $\gt$-structures are both stable within their families. 
%Therefore, using a $S^1$ action on $\overline{M}$, the authors can construct $3$-Sasakian structure  with $3$-orthonormal Killing vector fields $\lbrace \xi_i \rbrace_{i=1}^3$ satisfying $[\xi_i, \xi_j] = \epsilon_{ijk} \xi_k$. On $3$-sasakian manifolds using two disjoint $3$-parameter families of coclosed $\gt$-structures each include exactly one of the natural nearly parallel $\gt$-structures, the Laplacian coflow starting at the dual of any initial coclosed $\gt$-structure in either families converges, after rescaling, to the dual of the nearly parallel $\gt$-structure in that family. In particular, the nearly parallel $\gt$-structures are both stable within their families. 

In the other hand, when $M=G/H$ is a homogeneous space and the solutions of \eqref{eq:Laplacian_coflow_intro} are required to be $G$-invariant, the Laplacian coflow becomes an ordinary differential equation. Namely, let $\fg$ and $\fh$ be the Lie algebras of $G$ and $H$ respectively, and $\fg=\fh\oplus\fm$ a \emph{reductive decomposition} (i.e. $\Ad(H)$-invariant), any $G$-invariant solution of \eqref{eq:Laplacian_coflow_intro} on $M$ is determined by an $\Ad(K)$-invariant $4$-form $\psi(t)$ on $\fm\simeq T_oM$ (where $o=1_GH$). Then, since $\Delta\psi$ is invariant by diffeomorphisms of $M$, the flow \eqref{eq:Laplacian_coflow_intro} restricted to $G$-invariant solutions is equivalent with:
\begin{align}\label{eq: invariant_Laplacian_coflow}
    \frac{d}{dt}\psi(t)=\Delta_{\psi(t)}\psi(t) \qforq \psi(t)\in \left(\Lambda^4\fm^*\right)^{\Ad(H)}.
\end{align}
%where $\fg$ and $\fh$ are the Lie algebras of $G$ and $H$, respectively, and $\left(\Lambda^4(\fg/\fh)^*\right)^{\Ad(H)}$ denotes the vector space of $\Ad(H)|_{\fg/\fh}$-invariant $4$-forms on $\fg/\fh$.
Hence, short time existence and uniqueness of \eqref{eq: invariant_Laplacian_coflow} are followed by the well known ODE arguments, since the linear map $\Delta$ on $\Lambda^4\fm^*$ is continuous. For instance, in \cite{kath2021}, Kath and Lauret obtained expanding solitons and immortal solutions of the Laplacian coflow when $M$ is the connected and simply connected Lie group with Lie algebra $\fa\ltimes \R^4$, where $\fa$ is any maximal $\R$-split torus of $\fsl(\R^4)$. The latest have been obtained using the bracket flow approach (see \cite{Lauret2016} for a deep exposition of this method). Conversely, using a direct method, Bagaglini, Fernández and Fino obtained explicit immortal solutions of \eqref{eq: invariant_Laplacian_coflow} when $M$ is the $7$-dimensional Heisenberg group \cite{bagaglini2020}. In \cite{Fino2018}, Bagaglini and Fino gave explicit immortal solutions and solitons of the Laplacian coflow for a subclass of almost Abelian Lie groups. 

In this work, we study the Laplacian coflow of invariant coclosed $\gt$-structures. Intending to proceed with this, in Section \ref{Sec: preliminaries}, we provide some preliminaries on coclosed $\gt$-structures to establish the notation that is going to be used for the rest of the paper. In Section \ref{Sec: Laplacian coflow}, we recall the definition of the Laplacian coflow of coclosed $\gt$-structures and its soliton solutions. Specifically for the parameter $\lambda\in \R$ and the vector field $X\in \sX(M)$, such that $\varphi$ satisfies the soliton equation \eqref{Eq.Soliton}. We characterise the soliton condition in terms of the full torsion tensor $T$ of $\varphi$ by
\begin{equation}\label{eq: soliton_eq_intro}
    \div T=-\frac{1}{2}(\Curl X)^\flat+X\lrcorner T \qandq
    -\Ric+\frac{1}{2}T\circ T+(\tr T)T=\frac{\lambda}{4}g+\frac{1}{2}\cL_X g.
\end{equation}
where  $T\circ T$ and $\div T$ are defined in \eqref{eq:products}, and \eqref{eq:div.curl}, respectively. $\Ric$ denotes the Ricci curvature induced by $\varphi$ and $(\Curl X)_c=\nabla_aX_b{\varphi^{ab}}_c$ denotes the curl of $X$. As an application of \eqref{eq: soliton_eq_intro}, we give in Corollary \ref{cor: non-compact_shrinking_soliton} an alternative proof for the non-existence of compact shriking solitons of the Laplacian coflow.  Finally, in Section \ref{Sec: almost abelian}, we address the Laplacian coflow of invariant coclosed $\gt$-structures on almost Abelian Lie groups $G_A$, with Lie algebra $\fg$ and Lie bracket $A\in \fg\fl(\R^6)$. Using the bracket flow, we write the Laplacian coflow \eqref{eq: invariant_Laplacian_coflow} in terms of the Lie bracket as
\begin{align}\label{eq: bracket flow introduction}
	\frac{d}{dt}A=&-\Big(\frac{1}{2}\tr(S_A)^2+\frac{1}{4}(\tr JA)^2\Big)A+\frac{1}{2}[A,[A,A^t]]+\frac{1}{2}[A,S_A\circ_6S_A],
	\end{align}
where $S_A$ denotes the symmetric part of $A$, $J$ is the canonical almost complex structure on $\R^6$ and the product $S_A\circ_6S_A$ is defined in \eqref{TcircT}. Hereby, we are able to prove:
\begin{theorem}\label{inmortal solution}
   Let $\fg$ be an almost Abelian Lie algebra with Lie bracket $A\in \fg\fl(\R^6)$ and coclosed (non-flat) $\gt$-structure $\varphi$. Then, the solution $\varphi(t)$  of the Laplacian coflow starting at $\varphi$   is immortal, i,e, it is defined for all $t\in (\varepsilon_1,\infty)$.
\end{theorem}
 In spite of not obtaining explicit solutions of \eqref{eq: invariant_Laplacian_coflow} as it has been done in \cite{Fino2018} for a subclass of almost Abelian Lie algebras, the Theorem \eqref{inmortal solution} generalise the result of long-time existence of solutions for any almost Abelian Lie algebra. Moreover, the ODE \eqref{eq: bracket flow introduction} allows to study the dynamical behavior of the $2$-parameter family 
 \begin{equation*}
	    A=\left[
	    \begin{array}{c|c}
	    B & 0 \\ \hline
	    0 & -B^t
	    \end{array}
	    \right] \qwithq B=\left[\begin{array}{ccc}
	    0 & x & 0 \\ 
	    y & 0 & 0 \\
	    0 & 0 & 0
	    \end{array}
	    \right] \qandq x,y\in\R
	\end{equation*}
 showing in Proposition \ref{Prop. Example} that the Laplacian coflow is stable.
 To conclude, we study the invariant solitons of the Laplacian coflow in terms of the Lie brackets $A\in \fg\fl(\R^6)$, satisfying the time independent equation (see Theorem \ref{alge_soliton_prop})
\begin{equation}\label{eq:bracket soliton intro}
		[A,A^t]+S_A\circ_6 S_A=-\Big(\tr S_A^2-\frac{1}{2}(\tr JA)^2+2d\Big)I_6+(D+D^t)|_{\bR^6},
 	\end{equation}
    where $D$ is a derivation of $(\fg,A)$ and
    $$
      d=\frac{|[A,A^t]|^2+\langle S_A\circ_6 S_A, [A,A^t]\rangle}{2|A|^2}.
    $$
    As an application of \eqref{eq:bracket soliton intro}, firstly we prove if $A$ is skew-symmetric, then $(\fg,A,\varphi)$ defines an invariant soliton of the Laplacian coflow (see Corollary \ref{skew_sym_algebraic_soliton}). Secondly, we prove that any invariant (non-flat) soliton on an almost Abelian Lie group is an expanding solition (Proposition \ref{prop_absences_solitons_steady}). Finally, as far as we know it, we provide the first example of a semi-algebraic soliton of the Laplacian coflow, which is no algebraic (Example \ref{ex: semi_algebraic_example}).
 
\subsubsection*{Note}
Fino and Bagaglini \cite{Fino2018} have substantial overlap with this paper. However, while a number of conclusions are similar, the point of view on the Laplacian coflow is different. In this paper, we use the bracket flow introduced by Lauret in \cite{Lauret2016}, while in \cite{Fino2018}, a more traditional geometric flow approach is used. Both approaches are valuable and complementary, since they provide different perspectives on the same phenomenon. Since we are studying the same flow in the same space, we want to emphasise that this paper has different techniques, and both papers will give a better understanding of the Laplacian coflow.   
    
\subsubsection*{Notation}
 Let $(M, g)$ be a smooth oriented Riemannian $7$-manifold. We use the Einstein summation convention throughout. We compute in a local orthonormal frame, so all indices are subscripts and any repeated indices are summed over all values from $1$ to $7$.
 A differential $k$-form $\alpha$ on $M$ will be written as
$$
\alpha=\frac{1}{k !} \alpha_{i_1 i_2 \cdots i_k} \mathrm{~d} x^{i_1} \wedge \mathrm{d} x^{i_2} \wedge \cdots \wedge \mathrm{d} x^{i_k}
$$
in local coordinates $\left(x^1, \ldots, x^7\right)$, where $\alpha_{i_1 i_2 \cdots i_k}$ is completely skew-symmetric in its indices. With this convention, the interior product $\left.\partial_m\right\lrcorner \alpha$ of $\alpha$ with a coordinate vector field $\partial_m$ is the $(k-1)$-form
$$
\left.\partial_m\right\lrcorner \alpha=\frac{1}{(k-1) !} \alpha_{m i_1 i_2 \cdots i_{k-1}} \mathrm{~d} x^{i_1} \wedge \mathrm{d} x^{i_2} \wedge \cdots \wedge \mathrm{d} x^{i_{k-1}}.
$$
The metric $g$ on a Riemannian manifold $M$ induces a metric on $k$-forms, such that the inner product of $\alpha$ and $\beta$ is
$$
g(\alpha, \beta)=\frac{1}{k !} \alpha_{i_1 \cdots i_k} \beta_{j_1 \ldots j_k} g^{i_1 j_1} \ldots g^{i_k j_k}.
$$
The Levi-Civita connection associated to $g$ is denoted by $\nabla$, and its Christoffel symbols by $\Gamma_{i j}^k$. We write $\nabla_i$ for covariant differentiation in the $\partial_i$ direction. If $T_{i_1 \cdots i_k}$ is a tensor of type $(0, k)$, then $\nabla_m T_{i_1 \cdots i_k}$ always means $\left(\nabla_m T\right)_{i_1 \cdots i_k}$. We write the exterior derivative $\mathrm{d} \alpha$ of a $k$-form $\alpha$ as
$$
\mathrm{d} \alpha=\frac{1}{k !}\left(\nabla_m \alpha_{i_1 \cdots i_k}\right) \mathrm{d} x^m \wedge \mathrm{d} x^{i_1} \cdots \wedge \mathrm{d} x^{i_k}
$$
in terms of the covariant derivative.
The metric $g$ defines an isomorphism between $TM$ and $T^*M$ (raising and lowering indices.) If $v$ is a vector field, then the metric dual 1-form $v^\flat$ is defined by $v^\flat(w)=g(v, w)$. In coordinates, $\left(\partial_i\right)^\flat=g_{i k} \mathrm{~d} x^k$. Similarly, the 1-form $\alpha$ has a metric dual vector field $\alpha^{\sharp}$, and $\left(\mathrm{d} x^i\right)^{\sharp}=g^{i k} \partial_k$.

We use '$\vol$' to denote the volume form on $M$ associated to the metric $g$ and an orientation. The Hodge star operator $*$ taking $k$-forms to $(7-k)$-forms is defined by
$$
\alpha \wedge * \beta=g(\alpha, \beta) \mathrm{vol}
$$
 Our convention for labelling the Riemann curvature tensor is
$$
R_{i j k m} \frac{\partial}{\partial x^m}=\left(\nabla_i \nabla_j-\nabla_j \nabla_i\right) \frac{\partial}{\partial x^k}
$$
in terms of coordinate vector fields. With this convention, the Ricci tensor is $R_{j k}=$ $R_{l j k l}$ and the first Bianchi identity of the Riemann curvature tensor is:
\begin{equation}\label{eq: 1-bianchi identity}
    R_{abmn}+R_{amnb}+R_{anbm}=0.
\end{equation}
%and the Ricci identity is
%$$
%\nabla_k \nabla_i X_l-\nabla_i \nabla_k X_l=-R_{k i l m} X_m .
%$$\todo{Nosotros usamos la identidad de Ricci y la segunda identidaad de Bianchi en algun lugar? Yo  recuerdo haber usado la primera identidad de Bianchi.}
 %We also have the Riemannian second Bianchi identity
%$$
%\nabla_i R_{j k a b}+\nabla_j R_{k i a b}+\nabla_k R_{i j a b}=0,
%$$
%which when contracted on $i, a$ gives
%$$
%\nabla_i R_{i b j k}=\nabla_k R_{j b}-\nabla_j R_{k b}
%$$
 %
 We use $\Gamma(E)$ to denote the space of smooth sections of $E$. As special instances, we denote  the following cases as:
\begin{itemize}
    \item $\Omega^k:=\Gamma\left(\Lambda^k\left(T^* M\right)\right)$ is the space of smooth $k$-forms on $M$;
    \item $\mathcal{S}:=\Gamma\left(\mathrm{S}^2\left(T^* M\right)\right)$ is the space of smooth symmetric 2-tensors on $M$.
    \item $\sX(M):=\Gamma(TM)$ the space of vector fields.
\end{itemize}
With respect to the metric $g$ on $M$, we use $\mathcal{S}_0$ to denote those sections $h$ of $\mathcal{S}$ that are traceless. That is, $\mathcal{S}_0$ consists of those sections of $\mathcal{S}$, such that $\operatorname{Tr} h=g^{i j} h_{i j}=0$ in local coordinates. Then $\mathcal{S} \simeq \Omega^0 \oplus \mathcal{S}_0$, where $h \in \mathcal{S}$ is decomposed as $h=\frac{1}{7}(\operatorname{Tr} h) g+h_0$. Then, we have $\Gamma\left(T^* M \otimes T M\right)=\Omega^0 \oplus \mathcal{S}_0 \oplus \Omega^2$, where the splitting is pointwise orthogonal with respect to the metric on $T^* M \otimes T M$ induced by $g$. 

\smallskip 

\noindent {\bf Acknowledgements:} The authors would like to thank Jorge Lauret for
introducing them to the idea of the bracket flow in this context, and also the Universidad Nacional de Cordoba for hosting that conversation in 2019. Also, we are grateful to  Henrique S{\'a} Earp for the meaningful discussions and advises. This work stems on the MATHAMSUD Regional Program 21-MATH-06 collaborations.  AM was funded by the Sao Paulo Research Foundation (Fapesp) [2021/08026-5] and JPS was supported by the Coordination for the Improvement of Higher Education Personnel-Brazil (CAPES) [88887.648550/2021-00].

\section{Preliminaries} \label{Sec: preliminaries}
In this section we collect some results related to   $\gt$-structures that will be needed in the present paper. Any result of this section can be found in \cite{Karigiannis2007, Grigorian2013, Bryant2006}. 
\subsection{$\gt$-structures and their torsion } A \emph{$\gt$-structure} on a $7$-manifold $M$ is given by a differential $3$-form $\varphi$ on $M$, which is pointwise isomorphic to  the $3$-form 
$$ 
\varphi_0
=e^{123}+e^{145}+e^{167}+e^{246}-e^{257}-e^{347}-e^{356}
\in\Lambda^3(\R^7)^*,
$$ 
where  $e^{ijk}=e^i\wedge e^j\wedge e^k$ and $\{e^1,\dots,e^7\}$ is the dual basis of the canonical basis of $\R^7$. %It is equivalently determined by a $3$-form  $\varphi$ on $M$ such that $(TM,\varphi)$ is pointwise isomorphic to $(\R^7,\varphi_0)$. 
The $\gt$-structure $\varphi$ determines a Riemannian metric $g_{\varphi}$ and a volume form $\vol_{\varphi}$ so that
$$
	6g_\varphi(X,Y)  \vol_{\varphi}
	=(X\lrcorner \varphi)\wedge (Y\lrcorner\varphi)\wedge\varphi \qforq X,Y\in \sX(M).
$$
In addition, $\varphi$ induces a Hodge star operator $\ast_{\varphi}$ and we denote its dual $4$-form by $\psi=\ast_{\varphi}\varphi$.    For simplicity, we will write $g=g_{\varphi}$ and $\ast=\ast_\varphi$.  A $\gt$-structure gives rise to a decomposition of the space of differential $k$-forms $\Omega^k$ on $M$ into irreducible $\gt$-submodules. For instance,
\begin{align}
\label{eq:form.decomp}
	\Omega^2 &= \Omega^2_7\oplus\Omega^2_{14}\qandq
	\Omega^3  = \Omega^3_1\oplus\Omega_{7}^{3}\oplus\Omega^3_{27},
\end{align}
where $\Omega^k_l$ has (pointwise) dimension $l$. %and this decomposition is orthogonal with respect to the metric $g$.  Via the Hodge star, this defines corresponding decompositions of the 4-forms  $\Omega^4$ and 5-forms $\Omega^5$.  
In \cite{Bryant2006}, R. Bryant defines an injective map $\ri_{\varphi}:\cS^2\to\Omega^3$, given in local coordinates $x^1,\dots x^7$ by %we define a linear operator $ \rj_{\varphi}:\Omega^3\rightarrow S^2$ by
%\begin{align}
%\label{Eq:j.operator}
	%\rj_{\varphi}(\gamma)(X,Y)
	%&=\ast_{\varphi}((X\lrcorner \varphi)\wedge (Y\lrcorner \varphi)\wedge\gamma),
%\end{align}
%where $X,Y$ are tangent vectors on $M$.  Then $\rj_{\varphi}$ is surjective with kernel equal to $\Omega^3_7$ (see e.g.~\cite{Karigiannis2007}*{Proposition 2.17}). We can also define an injective linear operator $\ri_{\varphi}:S^2\to\Omega^3_1\oplus\Omega^3_{27}$ as in \cite{Bryant2006} which is (up to scaling) a right inverse for $\rj_{\varphi}$, and is given locally using summation notation by
	\begin{align}
	   	\ri_{\varphi}(h)&=\frac{1}{3!}\rivarphi(h)_{ijk}\rd x^{ijk}=\frac{1}{3!}(h_i^m\varphi_{mjk}+h_j^m\varphi_{imk}+h_k^m\varphi_{ijm})\rd x^{ijk}
	   	,\label{Eq:i.operator} 
	\end{align}
	where  $h\in \cS^2$ is a symmetric $2$-tensor field on $M$. % is given locally by $h_{ij}e^ie^j$.  We note that $\ri_{\varphi}(g)=3\varphi$, $\rj_{\varphi}(\varphi)=6g$ and
	%\begin{equation}\label{eq:i.j}
	  %\rj_{\varphi}\circ\ri_{\varphi}(h)=2(\tr h)g+4h,  
	%\end{equation}
	%where $\tr h=g^{ij}h_{ij}$ is the trace of $h$ with respect to $g$. 
  Additionally, the map $\rivarphi$ is surjective on $\Omega^3_1\oplus\Omega^3_{27}$ and its Hodge dual satisfies (e.g. \cite[Proposition 2.8]{Karigiannis2007})
    \begin{equation}\label{eq:i_psi}
        \ast\ri_{\varphi}(h)=\frac{1}{4!}(\Bar{h}_i^m\psi_{mjkl}+\Bar{h}_j^m\psi_{imkl}+\Bar{h}_k^m\psi_{ijml}+\Bar{h}_l^m\psi_{ijkm})\rd x^{ijkl}=:\ripsi(\Bar{h}),
    \end{equation}
    where $\Bar{h}=\frac14\tr(h)g-h$. In particular, for any trace-free symmetric 2-tensors $h\in S_0^2$, we have  $\ri_{\varphi}(h)\in\Omega^3_{27}$ and $\ripsi(h)\in \Omega^4_{27}=\ast\left(\Omega^3_{27}\right)$.
According with the $\gt$-decomposition of $\Omega^4$ and $\Omega^5$, the exterior derivative of $\varphi$ and $\psi$ are completely described in term of the \emph{torsion forms} $\tau_0\in\Omega^0$, $\tau_1\in\Omega^1$,  $\tau_2\in\Omega_{14}^2$ and $\tau_3\in\Omega^3_{27}$,  given by (see ~\cite{Bryant2006}*{Proposition 1})
\begin{align}\label{eq: Fernandez dpsi}
  	\rd\varphi 
  	=\tau_0\psi+ 3\tau_1\wedge\varphi+\ast\tau_3  \in  \Omega_1^4\oplus\Omega^4_7\oplus\Omega^4_{27} \qandq
	\rd\psi 
=4\tau_1\wedge\psi+\tau_2\wedge\varphi \in  \Omega^5_7\oplus\Omega^5_{14}.  
\end{align}
%Together, the forms $\lbrace \tau_0,\tau_1,\tau_2,\tau_3 \rbrace$ are called the \emph{intrinsic torsion forms} of the $\gt$-structure $\varphi$.
Moreover, for the \emph{full torsion tensor} $T$ of $\varphi$, which is defined locally by (see \cite{Karigiannis2007})
\begin{equation}
\label{Eq.nabla.varphi}
	\nabla_i\varphi_{jkl}
	=T_i^m\psi_{mjkl},
\end{equation}
we may use the torsion forms to relate them with $T$, by
\begin{equation}
\label{Eq:Torsion}
	T =\frac{\tau_0}{4}g -\tau_{27}-\tau_1^{\sharp}\lrcorner\varphi -\frac{1}{2}\tau_2 ,
\end{equation}
where $\tau_{27}$ is the trace-free symmetric $2$-tensor satisfying $\tau_3=\rivarphi(\tau_{27})$ and $\tau_1^\sharp$ denotes the unique vector field induced by the Riemannian metric $g$, (i.e. $g(\tau_1^\sharp,X)=\tau_1(X)$ for any $\in \sX(M)$). In addition, from \eqref{Eq.nabla.varphi} for the $4$-form $\psi$, we have
\begin{equation}\label{eq: nabla.psi}
    \nabla_m\psi_{ijkl}=-(T_{mi}\varphi_{jkl}-T_{mj}\varphi_{ikl}-T_{mk}\varphi_{jil}-T_{ml}\varphi_{jki}).
\end{equation}
	
\subsection{Properties of coclosed $\gt$-structures}
A $\gt$-structure $\varphi$ is \emph{coclosed} if it satisfies $\rd\psi=0$, in terms of \eqref{eq: Fernandez dpsi} the coclosed condition is equivalent with $\tau_1=0$ and $\tau_2=0$. Hence, the full torsion tensor of a coclosed $\gt$-structure simplifies to the symmetric $2$-tensor
\begin{equation}
\label{eq:torsion.coclosed}
    T=\frac{\tau_0}{4}g -\tau_{27} 
    \in \cS^2.   
\end{equation}
In addition, $\rd\varphi\in \Omega^4_1\oplus\Omega^4_{27}$ thus, by \eqref{Eq:i.operator}, \eqref{eq: Fernandez dpsi} and \eqref{eq:torsion.coclosed}, we have
\begin{equation}\label{eq: dvarphi_i_h}
    \rd\varphi=\ast\rivarphi\left(\frac{1}{3}(\tr T)g-T\right).
\end{equation}
The following proposition include some well known identities of coclosed $\gt$-structures given in \cite{Grigorian2013}, obtained  as consequence of a general formulae of the exterior derivative of a generic $3$-form. Here, we give an alternative proof of those identities, using the called \emph{$\gt$-Bianchi type identity}
\begin{equation}\label{eq: Bianchi-type identity}
    \nabla_iT_{jk}-\nabla_jT_{ik}=\left(\frac{1}{2}R_{ijmn}-T_{im}T_{jn}\right){\varphi_k}^{mn},
\end{equation}
where $T_{ij}$ is the coordinate of \eqref{eq:torsion.coclosed} and $R_{ijmn}$ denotes the Riemann curvature tensor. We remark that the identity \eqref{eq: Bianchi-type identity} can be read as the infinitesimal version of the diffeomorphism invariance of $T$ as a function $\varphi$. (see \cite{Karigiannis2007}*{Section 4} for an extensive discussion in the $\gt$-case and \cite{fadel2022} for any $H$-structure.). In the statement, for any $h,k\in \cS^2$, we denote the inner product $\langle h,k\rangle$ and the circ product $h\circ k\in \cS^2$ by
 	%The following proposition, adapted from \cite{Grigorian2013}*{Proposition 2.3, Lemma 4.5, (4.30), (5.8)}, provides useful properties of a coclosed $\gt$-structure and its full torsion tensor $T$.  In the statement, for $h,k\in S^2$ given locally by $h=h_{ij}dx^idx^j$ and $k=k_{ij}dx^idx^j$, we define the inner product $\langle h,k\rangle$ and the circ product $h\circ k\in S^2$ by
\begin{align}
\label{eq:products}
    \langle h,k\rangle = h_{ij}k_{ab}g^{ia}g^{jb}\qandq (h\circ k)_{ab}=\varphi_{amn}\varphi_{bpq}h^{mp}k^{nq}.
\end{align}
As well as the divergence and the curl of $h$, given in coordinates by
\begin{align}\label{eq:div.curl}
    \div h_a=\nabla_bh_a^b\qandq \Curl h_{ab}=\nabla_m h_{an}\varphi_b^{\,\,mn}. 
\end{align}

\begin{prop} 
Let $\varphi$ be a coclosed $\gt$-structure with full torsion tensor $T$, then the divergence and the curl of $T$ satisfy we have 
\begin{equation}\label{eq: Laplacian.Grigorian}
        \div T_a= \nabla_a\tr T \qandq \Curl T_{ab}=\Curl T_{ba}.
\end{equation}
In addition, the Ricci tensor and the scalar curvature are
\begin{equation}\label{eq: Ricci_formula}
      \Ric =-\Curl T-T^2+(\tr T)T \qandq
        R=(\tr T)^2-|T|^2.
\end{equation}
\end{prop}

\begin{proof}
Using \eqref{eq: 1-bianchi identity} and  the symmetries of $R_{abmn}$ is easy to prove that 
\begin{equation}\label{eq: contraction R_varphi}
    R_{abmn}\varphi^{bmn}=0 \qandq R_{amnp}{\psi_b}^{mnp}=0.
\end{equation}
Now, since $T$ is symmetric, using \eqref{eq: Bianchi-type identity} and \eqref{eq: contraction R_varphi} for the divergence $T$, we have
\begin{align*}
    \div T_a=\nabla_bT_a^b=\nabla_aT_b^b+\left(\frac12 R_{bamn}-T_{am}T_{bn}\right)\varphi^{bmn}=\nabla_a\tr T,
\end{align*}
and in addition, by \eqref{Eq.contraction.var.var. 1 index} for the curl of $T$, we get
\begin{align*}
    \Curl T_{ab}-\Curl T_{ba}=&\nabla_mT_{an}{\varphi_b}^{mn}-\nabla_mT_{bn}{\varphi_a}^{mn}\\
    =&\left(\frac12 R_{mapq}-T_{mp}T_{aq}\right){\varphi_n}^{pq}{\varphi_b}^{mn}-\left(\frac12 R_{mbpq}-T_{mp}T_{bq}\right){\varphi_n}^{pq}{\varphi_a}^{mn}\\
     =&\left(\frac12 R_{mapq}-T_{mp}T_{aq}\right)(g_b^pg^{qm}-g_b^qg^{pm}+{\psi_b}^{mpq})\\
     &-\left(\frac12 R_{mbpq}-T_{mp}T_{bq}\right)(g_a^pg^{qm}-g_a^qg^{pm}+{\psi_a}^{mpq})\\
     =&\frac12 R_{mabq}g^{mq}-\frac12 R_{mapb}g^{mp}-T_{am}T^m_b+\tr (T)T_{ab}\\
     &-\frac12 R_{mbaq}g^{mq}+\frac12 R_{mbpa}g^{mp}+T_{bm}T^m_a-\tr (T)T_{ba}\\
     =&-\Ric_{ab}+\Ric_{ba}=0.
\end{align*}

The formula for $\Ric$ can be derived from the computation above and for the scalar curvature, it follow from the observation 
$$
\Curl T_{aa}=\nabla_mT_{an}{\varphi_a}^{mn}=0.
$$
\end{proof}

Similar to \cite{Bryant2006}*{Corollary 2} for the case of closed $\gt$-structures, we can characterise the Einstein metrics induced by a coclosed $\gt$-structure:

\begin{cor}
    A coclosed $\gt$-structure $\varphi$ induces an Einstein metric if and only if the full torsion tensor satisfies
    \begin{equation}\label{eq: Einstein condition}
        \rivarphi(\Curl T)=\frac{3}{7}|T|^2\varphi-(\tr T)\tau_3-\rivarphi(T^2).
    \end{equation}
\end{cor}

\begin{proof}
    The result follows by applying the map $\rivarphi$ in \eqref{eq: Ricci_formula}.
\end{proof}

\begin{remark}
     Using the expression of the full torsion tensor in terms of the torsion forms \eqref{eq:torsion.coclosed}, the equation \eqref{eq: Einstein condition} becomes
    \begin{equation}\label{eq: Einstein_torsion_forms}
        \rivarphi(\Curl \tau_{27})=\frac37|\tau_{27}|^2\varphi-\frac{5\tau_0}{4}\tau_3-\rivarphi(\tau_{27}^2).
    \end{equation}
    It is well know that a metric induced by the  nearly parallel $\gt$-structure (i.e. $\tau_3=\rivarphi(\tau_{27})=0$) is Einstein. It is easy to check that \eqref{eq: Einstein_torsion_forms} is satisfied trivially for a nearly $\gt$-structure. 
\end{remark}

%\begin{prop} 
%\label{Prop.i operator coclosed}
	%Suppose we have a coclosed $\ $\varphi$  on a manifold $M$ and recall the notation in \eqref{Eq:i.operator}, \eqref{eq:products} and \eqref{eq:div.curl}.
%\begin{itemize}
 %   \item[(a)]	  If $\mu=\ri_{\varphi}(h)\in\Omega^3_{27}$ with $h\in S^2_0$ then
	%\begin{align*}
		%\rd\mu
		%&=-\frac{1}{2}(\div h)^{\flat}\wedge\varphi +*\ri_{\varphi}(k)
	%\end{align*}
	%where
	%$$ 
	%k=\frac{1}{2}\big(\Curl h+(\Curl h)^{t}\big) 
	%+\frac{1}{2} T\circ h +\frac{1}{2}\big(Th+(Th)^{t}\big) -\frac{1}{2}(\tr T)h -\frac{1}{6}\langle T,h\rangle g.$$
	%\item[(b)] The full torsion tensor $T$ satisfies the following identities: 
%\begin{equation}
%\label{eq: Laplacian.Grigorian}
 %   \begin{array}{rcl}
  %      \div T=& \nabla(\tr T),\\
   %     \Ric =&-\Curl T-T^2+(\tr T) T, 
    %\end{array}
    %\begin{array}{rcl}
   %     \Curl T=& (\Curl T)^{t},\\
    %    R=&(\tr T)^2-|T|^2.
    %\end{array}
%\end{equation}
%]\end{itemize}	
%\end{prop}

 \section{Laplacian coflow of $\gt$-structures}\label{Sec: Laplacian coflow}
In this section, we recall  the definition of the Laplacian coflow and we also study soliton solutions  and symmetries of coclosed $\gt$-structure. For these we follow \cite{Karigiannis2012,Grigorian2013}.  
 
\begin{definition}
		A time-dependent $\gt$-structure $\{\varphi(t)\}_{t\in (\epsilon_1,\epsilon_2)}$ on a $7$-manifold $M$, satisfies the Laplacian coflow of coclosed $\gt$-strcutures, if for any $t\in(\epsilon_1,\epsilon_2)$ we have 
		\begin{equation}
		\label{Eq.coflowLaplacian}
		\frac{\partial}{\partial t}\psi(t)=\Delta_t\psi(t) \qandq \rd\psi(t)=0,
		\end{equation}
		where $\psi(t)=\ast_t\varphi(t)$ and $\Delta_t=\rd\rd^{\ast_t}+\rd^{\ast_t}\rd$ is the Hodge Laplacian with respect to the metric $g(t)=g_{\varphi(t)}$.  %We will always restrict to solutions so that $\varphi(t)$ is a coclosed $\gt$-structure for all $t$, i.e.~$\rd\psi(t)=0$ for all $t\in(\epsilon_1,\epsilon_2)$.
	\end{definition}
%    If $M^7$ is compact then the volume of $M$ determined by the $\gt$-structure $\varphi$ on $M$ is:
 %   \begin{equation}
  %     \label{Eq: functional volume} \mathcal{H}(\varphi):=\textrm{Vol}(M,\varphi)=\frac{1}{7}\int_M\phi\wedge \psi.
   % \end{equation}
	
%\begin{prop}\cite[Proposition 3.4 and \S4]{Grigorian2013}
 %   The flow \eqref{Eq.coflowLaplacian} for coclosed $\gt$-structures $\ast_{\varphi
  %  }\varphi$ is the gradient flow of the volume functional in \eqref{Eq: functional volume} restricted to $[\ast_{\varphi}\varphi]$ and the critical points are strict local maxima for the volume functional (modulo diffeomorphisms).  
%\end{prop}
%Therefore the evolution of the metric, the inverse metric and volume form are given in the following proposition:
	
%\subsection{Soliton solutions of Laplacian coflow}
As for many geometric flows, we are interested in considering \emph{self-similar solutions}
\begin{equation}\label{eq: self-similar_solution}
    \varphi(t)=\lambda(t)f(t)^*\varphi \qwhereq \lambda(t)\in C^\infty(M) \qandq f(t)\in \Diff(M),
\end{equation}
it means, solutions that evolves the initial data $\varphi$ by diffeomorphisms and scalings, since these kind of solutions are expected to be related to singularities of the flow. In particular, self-similar solutions with initial condition $\varphi$ are equivalent with a time independent equation of $\psi=\ast\varphi$, called the \emph{soliton equation}, namely,  $\varphi$ is called a \emph{soliton} for the Laplacian coflow \eqref{Eq.coflowLaplacian}, if $\psi$ satisfies the soliton equation: %for coclosed $\gt$-structures on $M$ is a triple $(\psi,X,\lambda)$ satisfying
	\begin{equation}\label{Eq.Soliton}
	\Delta_\psi\psi=\cL_{X}\psi+\lambda\psi%=\rd(X\lrcorner \psi)+\lambda\psi, 
	\end{equation} 
	where $\lambda\in\R$ and $X$ is a complete vector field on $M$. Moreover,  the soliton $(\varphi,\lambda,X)$ is called \emph{expanding, steady}, or \emph{shrinking}, if $\lambda>0$, $\lambda=0$ or $\lambda<0$, respectively.  %We are interested in $\gt$-structures satisfying \eqref{Eq.Soliton}. 

The following lemma, decomposes the Hodge Laplacian of $\psi$ according to the $\gt$-irreducible decomposition of $\Omega^2$, it appeared  originally in \cite{Grigorian2013}*{Proposition 4.6}. Here, we provide the computations in detail for the self-contained of the work, for it we follow the computation given in \cite{Lotay2017} for $\Delta_\varphi\varphi$ in the closed case. 

\begin{lemma}\label{Lemma. Laplacian Grigorian}
    Let $\varphi$ be a coclosed $\gt$-structure on a manifold $M$ with associated metric $g$. Then,
\begin{align*}
\Delta_{\psi}\psi&=\frac{2}{7}((\tr T)^2+|T|^2)\psi\oplus( 
    \rd\tr T) \wedge\varphi\\
	& \quad\oplus \ast_{\varphi}\ri_{\varphi}\Big(\Ric-\frac{1}{2}T\circ T-(\tr T)T
	+\frac{1}{14}\left((\tr T)^2+|T|^2\right)g\Big)\in \Omega^4_1\oplus\Omega_7^4\oplus\Omega^4_{27}.
\end{align*}
\end{lemma}

\begin{proof}
    Since $\rd\psi=0$, by  \eqref{eq: dvarphi_i_h} we have
    \begin{equation}\label{Eq: Laplacian psi exact}
        \Delta_\psi\psi=\rd\rd^*\psi=\rd\ast\rd\varphi=\rd\beta \qwhereq \beta:=\ri_{\varphi}(h)=\ri_{\varphi}\left(\frac{1}{3}(\tr T)g-T\right)\in \Omega^3_{27}.
    \end{equation}
    In local coordinates, we can write \eqref{Eq: Laplacian psi exact} as 
    $$
    \Delta_\psi\psi=\frac{1}{4!}(\Delta_\psi\psi)_{ijkl}dx^{ijkl},
    $$
    where
    \begin{equation}\label{eq: coordinates of Laplacian psi}
        (\Delta_\psi\psi)_{ijkl}=\nabla_i\beta_{jkl}-\nabla_j\beta_{ikl}+\nabla_k\beta_{ijl}-\nabla_l\beta_{ijk}.
    \end{equation}
    We can decomposes $\Delta_\psi\psi$ into irreducible summands as
    \begin{equation}
        \Delta_\psi\psi=a\psi + X^\flat\wedge\varphi+\ast\ri_{\varphi}(s),
    \end{equation}
    where $a$ is a function, $X$ a vector field and $s$ is a trace-less symmetric $2$-tensor. Now, we compute the expression of $a$, $X$ and $s$ in terms of the full torsion tensor of $\varphi$.
    For $a$, using \eqref{Eq: Laplacian psi exact}, \eqref{eq: coordinates of Laplacian psi}, \eqref{Eq.contraction.var.psi.2 index} and \eqref{Eq.contraction.psi.psi.2 index}, we have
    \begin{align*}
        a=\frac{1}{7}\langle\Delta_\psi,\psi\rangle=&\frac{1}{168}(\nabla_i\beta_{jkl}-\nabla_j\beta_{ikl}+\nabla_k\beta_{ijl}-\nabla_l\beta_{ijk})\psi^{ijkl}\\
        =&\frac{1}{42}\nabla_i(h_j^m\varphi_{mkl}+h_k^m\varphi_{jml}+h_l^m\varphi_{jkm})\psi^{ijkl}\\
        =&\frac{1}{14}(\nabla_ih_j^m\varphi_{mkl}+h_j^mT_i^n\psi_{nmkl})\psi^{ijkl}\\
        =&\frac{2}{7}(\nabla_ih_j^m{\varphi_m}^{ij}+\tr h \tr T-\langle h,T\rangle)\\
        =&\frac{2}{7}\left((\tr T)^2+|T|^2\right),
    \end{align*}
    where $h$ is the symmetric $2$-tensor given in \eqref{Eq: Laplacian psi exact}. For the vector field $X$, we have
    \begin{align*}
        \langle \Delta_\psi\psi,e^m\wedge\varphi\rangle=\ast(X^\flat\wedge\varphi\wedge\ast(e^m\wedge\varphi))=4\langle X^\flat, e^m\rangle=4X_ng^{nm}.
    \end{align*}
    Thus, using \eqref{Eq: Laplacian psi exact}, \eqref{Eq.contraction.var.var. 1 index}, \eqref{Eq.contraction.var.var.2 index}, \eqref{Eq.contraction.var.psi.2 index} and \eqref{eq: Laplacian.Grigorian}, we get
    \begin{align*}
        X_m=&\frac{1}{4}\langle\Delta_\psi\psi,e^n\wedge\varphi\rangle g_{mn}\\
        =&\frac{1}{98}(\nabla_i\beta_{jkl}-\nabla_j\beta_{ikl}+\nabla_k\beta_{ijl}-\nabla_l\beta_{ijk})(e^n\wedge\varphi)^{ijkl}g_{mn}\\
        =&\frac{1}{4!}(\nabla_m\beta_{jkl}\varphi^{jkl}-3\nabla_j\beta_{mkl}\varphi^{jkl})\\
        =&\frac{1}{4!}(\nabla_m(\beta_{jkl}\varphi^{jkl})-\beta_{jkl}\nabla_m\varphi^{jkl}-3\nabla_j(\beta_{mkl}\varphi^{jkl})+3\beta_{mkl}\nabla_j\varphi^{jkl})\\
        =&\frac{1}{4!}\left(3\nabla_m(h_j^n\varphi_{njk}\varphi^{jkl})-3h_j^n\varphi_{nkl}T_{mp}\psi^{pjkl}-3\nabla_j(h_m^n\varphi_{nkl}\varphi^{jkl}+2h_k^n\varphi_{mnl}\varphi^{jkl})\right)\\
        =&\frac{1}{8}\left(6\nabla_mh_j^ng^j_n-4h_j^nT_{mp}{\varphi_n}^{pj}-6\nabla_jh_m^j-2\nabla_m h_k^ng_n^k+2\nabla_nh_m^n\right)\\
        =&\frac{1}{2}\left(\frac{4}{3}\nabla_m(\tr T)-\frac{1}{3}\nabla_m(\tr T)+\nabla_jT_m^j\right)=(\div T)_m.
    \end{align*}
    Finally, to find the symmetric $2$-tensor $s$, we have:
    \begin{align}\label{Eq.symmetric.part.Laplacian}
        \begin{split}
	&  (\Delta_\psi\psi)_{imnp}{\psi_j}^{mnp}+(\Delta_\psi\psi)_{jmnp}{\psi_i}^{mnp}\\
	&  =a(\psi_{imnp}{\psi_j}^{mnp}+\psi_{jmnp}{\psi_i}^{mnp})+(\ast \ri_{\varphi}(s))_{imnp}{\psi_j}^{mnp}+(\ast \ri_{\varphi}(s))_{jmnp}{\psi_i}^{mnp},
	\end{split}
    \end{align}
    Then, using \eqref{eq:i_psi},\eqref{Eq.contraction.psi.psi.2 index} and \eqref{Eq.contraction.psi.ps. 3 index}, we get
 \begin{align*}
     (\ast \ri_{\varphi}(s))_{imnp}{\psi_j}^{mnp}=&-s_i^q\psi_{qmnp}{\psi_j}^{mnp}-3s_m^q\psi_{iqnp}{\psi_j}^{mnp}\\ \nonumber
     =&-24s_i^qg_{qj}-3s_m^q(4g_{ij}g_q^m-4g_i^mg_{qj}+2{\psi_{iqj}}^m)=-12s_{ij}
 \end{align*}
 By symmetry, the right hand side of  \eqref{Eq.symmetric.part.Laplacian} becomes
\begin{equation}
\label{Eq.Laplacian.h.left}	                    \begin{split}
(\Delta_\psi\psi)_{imnp}{\psi_j}^{mnp}+(\Delta_\psi\psi)_{jmnp}{\psi_i}^{mnp}=24\left(2ag_{ij}-s_{ij}\right).
	\end{split}
\end{equation}
Now, using \eqref{eq: coordinates of Laplacian psi}, \eqref{Eq.contraction.var.psi.2 index}, \eqref{Eq.contraction.var.psi.2 index}, \eqref{Eq.contraction.var..psi.1 index} and \eqref{Eq.contraction.psi.psi.2 index}, we have
\begin{align*}
    (\Delta_\psi\psi)_{imnp}{\psi_j}^{mnp}=&(\nabla_i\beta_{mnp}-3\nabla_m\beta_{inp}){\psi_j}^{mnp}\\
    =&3(\nabla_ih_m^q\varphi_{qnp}+h_m^q\nabla_i\varphi_{qnp}){\psi_j}^{mnp}-3\nabla_m((h_i^q\varphi_{qnp}+2h_n^q\varphi_{iqp}){\psi_j}^{mnp})\\
    &+3(h_i^q\varphi_{qnp}+2h_n^q\varphi_{iqp})\nabla_m{\psi_j}^{mnp}\\
    =&3\left(4\nabla_ih_m^q{\varphi_{qj}}^m+h_m^qT_i^l(4g_{lj}g_q^m-4g_l^mg_{qj}+2{\psi_{lqj}}^m)-4\nabla_m(h_i^q{\varphi_{qj}}^m)\right.\\
    &\left.+2\nabla_m(h_i^q{\varphi_{qj}}^m-h_{nj}{\varphi_i}^{mn}-h_n^m{\varphi_{ji}}^n-\tr h {\varphi_{ij}}^m)\right.\\
    &\left.+(h_i^q\varphi_{qnp}+2h_n^q\varphi_{iqp})(-T_{mj}\varphi^{mnp}+\tr T{\varphi_j}^{np}-T_m^n{\varphi_j}^{mp}+T_m^p{\varphi_j}^{mn})\right)\\
    =&6\left(2(\tr h T_{ij}-T_i^mh_{mj})-\nabla_mh_i^q{\varphi_{qj}}^m-\nabla_mh_{nj}{\varphi_i}^{mn}-\nabla_m(h_n^m{\varphi_{ji}}^n+\tr h {\varphi_{ij}}^m)\right.\\
    &\left.-3h_i^mT_{mj}-\tr h T_{ij}+h_i^mT_{mj}+3\tr T h_{ij}+\tr T\tr h g_{ij}-\tr T h_{ij}-\tr T h_{ij}+h_i^mT_{mj}\right.\\
    &\left.-T_m^nh_n^mg_{ij}+T_i^mh_{mj}-(T\circ h)_{ij}\right)\\
    =&6\left(\tr h T_{ij}-T_i^mh_{mj}-(\Curl h)_{ij}-(\Curl h)_{ji}-\nabla_m(h_n^m{\varphi_{ji}}^n+\tr h {\varphi_{ij}}^m)\right.\\
    &\left.-h_i^mT_{mj}+\tr T h_{ij}+(\tr T\tr h-\langle T,h\rangle)g_{ij}-(T\circ h)_{ij}\right).
\end{align*}
Thus, replacing $h=\frac{1}{3}(\tr T)g-T$ in the above expression and using \eqref{eq: Laplacian.Grigorian}, the left hand side of \eqref{Eq.symmetric.part.Laplacian} becomes
\begin{align*}
     (\Delta_\psi\psi)_{imnp}{\psi_j}^{mnp}+(\Delta_\psi\psi)_{jmnp}{\psi_i}^{mnp} =24\left(T_i^mT_{mj}+(\Curl T)_{ij}+\frac12((\tr T)^2+|T|^2)g_{ij}+\frac12(T\circ T)_{ij}\right).
\end{align*}
Finally, from \eqref{Eq.Laplacian.h.left}, we obtain 
$$
s_{ij}=-(\Curl T)_{ij}-T_i^mT_{mj}-\frac12(T\circ T)_{ij}+\frac{1}{14}((\tr T)^2+|T|^2)g_{ij}
$$
\end{proof}

Similar to the Laplacian of $\psi$, we can compute the decomposition of the Lie derivative with respect to any vector field. We recall that the vector field $X$ is called an \emph{infinitesimal symmetry} of $\psi$, if $\cL_X\psi=0$. The next result was done in \cite{dgk-isometric} for the $3$-form $\varphi$.    
\begin{prop}
\label{theorem.lie.derivative.psi}
	Let $\varphi$ be a coclosed $\gt$-structure on $M^7$, with associated metric $g$, and let $X$ be a vector field on $M$. Then, if $\psi=*\varphi$,  
	\begin{equation}
	\label{Eq:Lie.derivative.psi}
		\cL_X \psi=\frac{4}{7}(\div X)\psi\oplus(-\frac{1}{2}\Curl X+X\lrcorner T  )^{\flat}\wedge\varphi\oplus\ast \ri_{\varphi}\Big(\frac{1}{7}(\div X)g -\frac{1}{2}(\cL_Xg)
		%+[X\lrcorner\varphi, T]
		\Big)\in\Omega_1^4\oplus\Omega_7^4\oplus\Omega_{27}^{4}.
		\end{equation}
		%where $\ri_{\varphi}$ is given by  \eqref{Eq:i.operator}. 
		In particular, $X$ is an infinitesimal symmetry of $\psi$ if and only if $X$ is  a Killing vector field of $g$ and satisfies $\Curl(X)=2X\lrcorner T$.% with no curl, i.e.,
%$$ 
 %   \cL_X\psi=0\quad
  %  \Longleftrightarrow
   % \quad	\cL_Xg=0
	%\qandq 
	%\Curl(X)=2X\lrcorner T.
%$$
\end{prop}
\begin{proof}
	Since $\varphi$ is coclosed, i.e. $\rd\psi=0$, we have 
	$$ \cL_X\psi=\rd(X\lrcorner\psi)+X\lrcorner\rd\psi= \rd(X\lrcorner \psi).$$
	Let $\alpha=X\lrcorner\psi$, so that locally $\alpha_{ijk}=X^l\psi_{lijk}$ and 
	\begin{equation*}		%\label{Eq.Lie.derivative.index }
		(\cL_X\psi)_{ijkl}=(\rd\alpha)_{ijkl}=\nabla_i\alpha_{jkl}-\nabla_j\alpha_{ikl}+\nabla_k\alpha_{ijl}-\nabla_l\alpha_{ijk}.
	\end{equation*}
	Denoting by $\pi_l^k:\Omega^k\rightarrow \Omega_l^k$ the orthogonal projections, we decompose $\cL_X\psi$ as
	\begin{equation}
	\label{Eq:decomposition.lie.derivative.psi}
		\cL_X\psi=\pi_1^4(\cL_X\psi)+\pi_7^4(\cL_X\psi)+\pi_{27}^4(\cL_X\psi)=a\psi+W^{\flat}\wedge\varphi+\ast i_{\varphi}(h),
	\end{equation}
	where  $a\in\Omega^0$, 
	and 
	$h$ is a trace-free symmetric $2$-tensor on $M$. We compute $a$ as follows: 
\begin{equation}
\label{Eq:valor.a.lx.psi }
    \begin{split}
    a &= \frac{1}{7}\langle\cL_X\psi,\psi\rangle=\frac{1}{168}(\nabla_i\alpha_{jkl}-\nabla_j\alpha_{ikl}+\nabla_k\alpha_{ijl}-\nabla_l\alpha_{ijk})\psi^{ijkl}\\
	&= \frac{1}{42}\nabla_i\alpha_{jkl}\psi^{ijkl}=\frac{1}{42}\nabla_i(\alpha_{jkl}\psi^{ijkl})-\frac{1}{42}\alpha_{jkl}\nabla_i\psi^{ijkl}\\
	&= \frac{24}{42}\nabla_i(X^mg_{mi})-\frac{1}{42}X^m\psi_{mjkl}(\nabla_i\psi^{ijkl})
	=\frac{4}{7}\nabla_iX_i=\frac{4}{7}\div X,
    \end{split}
\end{equation}
	where we used \eqref{Eq.contraction.var.psi.2 index} and because $T$ is symmetric.  
	To compute $W^{\flat}$, note that 
	$$ \langle\ast((\ast \cL_X\psi)\wedge\varphi),e^m\rangle= 4\langle W^{\flat},e^m\rangle,$$ 
	thus
	\begin{align*}
	    4W^{m} &= \ast\big((\ast\cL_X\psi)\wedge\varphi\wedge e^m\big)=\langle\varphi\wedge e^m,\cL_X\psi\rangle
		= \langle \varphi\wedge e^m, d\alpha\rangle
	\end{align*}
	Therefore, we obtain
\begin{equation}
\label{Eq:W.lie.derivative.psi}
    \begin{split}
    W^m &= \frac{1}{4}\langle \varphi\wedge e^m,d\alpha
\rangle=\frac{1}{4!}(\nabla^i\alpha^{jkm}-\nabla^j\alpha^{ikm}+\nabla^k\alpha^{ijm}-\nabla^m\alpha^{ijk})\varphi_{ijk}\\
    &=\frac{1}{4!}(3\nabla^i\alpha^{jkm}\varphi_{ijk}-\nabla^m\alpha^{ijk}\varphi_{ijk})\\
    &=\frac{3}{4!}\nabla^i(\alpha^{jkm}\varphi_{ijk})-\frac{3}{4!}\alpha^{jkm}\nabla^i\varphi_{ijk}-\frac{1}{4!}\nabla^m(\alpha^{ijk}\varphi_{ijk})+\frac{1}{4!}\alpha^{ijk}\nabla^{m}\varphi_{ijk}\\
     &=\frac{3}{4!}\nabla^i(X_l\psi^{ljkm}\varphi_{ijk})-\frac{3}{4!}X_l\psi^{ljkm}T^i_n\psi^n_{\,\,ijk}-\frac{1}{4!}\nabla^m(X_l\psi^{lijk}\varphi_{ijk})+\frac{1}{4!}X_l\psi^{lijk}T_n^m\psi^n_{\,\,ijk}\\
     &=-\frac{1}{2} \nabla^i(X_l\varphi_i^{\,\,lm})+X_lT^{ml}=-\frac{1}{2}(\nabla^iX_l\varphi_i^{\,\,lm}+X_l\nabla_i\varphi_i^{\,\,lm})+X_lT^{ml}\\
     &=-\frac{1}{2}\Curl X^m-\frac{1}{2}X_lT_i^{\,\,n}\psi_{ni}^{\,\,\,\,lm}+(X\lrcorner T)^m=-\frac{1}{2}(\Curl X)^m+(X\lrcorner T)^m
    \end{split}
\end{equation}
	Finally, to compute $h$, observe that
\begin{equation}		\label{Eq.Soliton.com.h}
	\begin{split}
	&  (\cL_X\psi)_{imnp}{\psi_j}^{mnp}+(\cL_X\psi)_{jmnp}{\psi_i}^{mnp}\\
	&  =a(\psi_{imnp}{\psi_j}^{mnp}+\psi_{jmnp}{\psi_i}^{mnp})+(\ast \ri_{\varphi}(h))_{imnp}{\psi_j}^{mnp}+(\ast \ri_{\varphi}(h))_{jmnp}{\psi_i}^{mnp},
	\end{split}
\end{equation}
	where $$(\ast \ri_{\varphi}(h))_{imnp}=-(h_i^q\psi_{qmnp}+h_m^q\psi_{iqnp}+h_n^q\psi_{imqp}+h_{p}^q\psi_{imnq}).$$
 Using \eqref{Eq.contraction.psi.psi.2 index} and \eqref{Eq.contraction.psi.ps. 3 index}, we get
 \begin{align*}
     (\ast \ri_{\varphi}(h))_{imnp}{\psi_j}^{mnp}=&-h_i^q\psi_{qmnp}{\psi_j}^{mnp}-3h_m^q\psi_{iqnp}{\psi_j}^{mnp}\\
     =&-24h_i^qg_{qj}-3h_m^q(4g_{ij}g_q^m-4g_i^mg_{qj}+2{\psi_{iqj}}^m)=-12h_{ij}
 \end{align*}
 % \todo[inline]{Creo que la formula correcta seria 
 % $$(\ast \ri_{\varphi}(h))_{mnpi}%=(\ri_{\psi}(h))_{mnpi}
	%	=-(h_m^q\psi_{qnpi}+h_n^q\psi_{mqpi}+h_p^q\psi_{mnqi}+h_{i}^q\psi_{mnpq}).$$ Viendo la prop 2.8 \cite{Karigiannis2007} o incluso tomando la estrella de %\eqref{Eq:i.operator}}
	By symmetry, the right hand side of  \eqref{Eq.Soliton.com.h} becomes
\begin{equation}
\label{Eq.soliton.h.left}	                    \begin{split}
(\cL_X\psi)_{imnp}{\psi_j}^{mnp}+(\cL_X\psi)_{jmnp}{\psi_i}^{mnp}=24\left(\frac87(\div X)g_{ij}-h_{ij}\right)
%   	&  a\psi_{mnpi}\psi_j^{mnp} +(h_m^q\psi_{qnpi} +h_{n}^q\psi_{mqpi} +h_{p}^q\psi_{mnqi} +h_i^q\psi_{mnpq})\psi_j^{mnp}\\
	%	&+ a\psi_{mnpj}\psi_i^{mnp} +(h_m^q\psi_{qnpj} +h_{n}^q\psi_{mqpj} +h_{p}^q\psi_{mnqj}h_j^q\psi_{mnpq})\psi_i^{mnp}\\
		%&= -48ag_{ij} +h_m^q(4g_{qj}g_{im} -4g_{qm}g_{ij} +4g_{qi}g_{jm} -4g_{qm}g_{ij})\\
		%&\quad  -h_n^q(4g_{qj}g_{in} -4g_{qn}g_{ij}+4g_{qi}g_{jn} -4g_{qn}g_{ij})+24h_i^qg_{qj}\\
		%& \quad +h_p^q(4g_{qj}g_{ip} -4g_{qp}g_{ji} +4g_{qi}g_{jp} -4g_{qp}g_{ij})+24h_j^qg_{qi}\\
		%&= -48ag_{ij}+ 56h_{ij}.
	\end{split}
\end{equation}
For the left-hand side of \eqref{Eq.soliton.h.left}, using the identities \eqref{Eq.contraction.var..psi.1 index},\eqref{Eq.contraction.var.psi.2 index},\eqref{Eq.contraction.psi.psi.2 index} and \eqref{Eq.contraction.psi.ps. 3 index}, we have:
\begin{align*}
    (\cL_X\psi)_{imnp}{\psi_j}^{mnp}=&\nabla_i\alpha_{mnp}{\psi_j}^{mnp}-3\nabla_m\alpha_{inp}{\psi_j}^{mnp}\\
    =&\nabla_i(\alpha_{mnp}{\psi_j}^{mnp})-\alpha_{mnp}\nabla_i{\psi_j}^{mnp}-3\nabla_m(\alpha_{inp}{\psi_j}^{mnp})+3\alpha_{inp}\nabla_m{\psi_j}^{mnp}\\
    =&24\nabla_iX_j-12T_i^m(X\lrcorner \varphi)_{mj}+12(\div X)g_{ij}-12\nabla_iX_j-6{(\nabla_mX\lrcorner\psi)_{ij}}^m\\
    &-6\tr(T)(X\lrcorner\varphi)_{ij}+6(X\lrcorner T)_m{\varphi^m}_{ij}+6T_i^m(X\lrcorner\varphi)_{mj}-6T_j^m(X\lrcorner\varphi)_{mi}\\
     &-12\tr(T)(X\lrcorner\varphi)_{ij}+6T_j^m(X\lrcorner\varphi)_{mi}+6T_i^m(X\lrcorner\varphi)_{mj}+6(X\lrcorner T)_m{\varphi^m}_{ij}\\
     =&12\nabla_iX_j+12(\div X)g_{ij}-6{(\nabla_mX\lrcorner\psi)_{ij}}^m-18\tr(T)(X\lrcorner\varphi)_{ij}+12(X\lrcorner T)_m{\varphi^m}_{ij}
\end{align*}
By symmetry, we get 
\begin{align*}
    (\cL_X\psi)_{imnp}{\psi_j}^{mnp}+(\cL_X\psi)_{jmnp}{\psi_i}^{mnp}=12(\nabla_iX_j+\nabla_jX_i)+24(\div X)g_{ij}
\end{align*}
	So, using \eqref{Eq:valor.a.lx.psi }, \eqref{Eq.soliton.h.left} and the above expressions, we obtain 
	$$	\frac{1}{2}(\nabla_iX_j+\nabla_jX_i)+(\div X)g_{ij}=\frac{8}{7}(\div X)h_{ij}-h_{ij}$$
	which, upon re-arranging, it gives
\begin{equation}		\label{Eq:h.lie.derivative.psi}
    \begin{split}
 	h_{ij } =\frac{1}{7}(\div X)g_{ij}-\frac{1}{2}(\cL_Xg)_{ij}
	\end{split}
\end{equation}
    Hence, substituting \eqref{Eq:valor.a.lx.psi }, \eqref{Eq:W.lie.derivative.psi} and \eqref{Eq:h.lie.derivative.psi} into \eqref{Eq:decomposition.lie.derivative.psi} we obtain \eqref{Eq:Lie.derivative.psi}.
	\end{proof}

\begin{prop}
\label{prop.general.soliton}
	Let $\varphi$ be a coclosed $\gt$-structure on $M$ with associated metric $g$. If $(\varphi,X,\lambda)$ is a soliton of the Laplacian coflow as in \eqref{Eq.Soliton}, then its full torsion tensor $T$ satisfies 
\begin{equation}
\label{Eq:Condition.soliton}
\begin{split}
    \div T=&-\frac{1}{2}(\Curl X)^\flat+X\lrcorner T,\\
    -\Ric+\frac{1}{2}T\circ T+(\tr T)T=&\frac{\lambda}{4}g+\frac{1}{2}\cL_X g.
		%-[X\lrcorner\varphi, T].
\end{split}
	\end{equation}
	\end{prop}
	\begin{proof}
	Using \eqref{eq: Laplacian.Grigorian}, \eqref{eq:i_psi} and Lemma \ref{Lemma. Laplacian Grigorian}, we obtain
		\begin{align*}
	\Delta_{\psi}\psi&=(\div T)^\flat\wedge\varphi+ \ri_{\psi} \Big(-\Ric+\frac{1}{2}T\circ T+(\tr T)T\Big).
	\end{align*}
On the other hand, by \eqref{eq:i_psi} and Proposition \ref{theorem.lie.derivative.psi} we have 
\begin{align*}
  \lambda\psi+\cL_X\psi=(-\frac{1}{2}\Curl X+X\lrcorner T  )^{\flat}\wedge\varphi+ \ri_{\psi}\Big(\frac{\lambda}{4}g +\frac{1}{2}(\cL_Xg)
		\Big)
\end{align*}
and thus we get \eqref{Eq:Condition.soliton}.
\end{proof}
	
\begin{remark}
   \begin{itemize}
       \item We notice that \eqref{Eq:Condition.soliton} coincides with the soliton equation for a general geometric flow given in \cite{fadel2022}*{Definition 1.52}.
       \item The tuple $(g,X,\lambda)$ is called a \emph{Ricci soliton} if it satisfies $\Ric=\lambda g+\cL_X g$. The second equation of \eqref{Eq:Condition.soliton} can be view as a perturbation of the Ricci soliton equation using the torsion tensor $T$. A similar remark was done by Lotay-Wei for the Laplacian flow \cite{Lotay2017}, but in contrast, the first equation of \eqref{Eq:Condition.soliton} coincides with one of the equation of the isometric soliton condition of the harmonic flow of $\gt$-structures \cite{Karigiannis2020}*{Definition 2.16}.
       \item From the second equation of \eqref{Eq:Condition.soliton} is natural to ask for solitons of the Laplacian coflow, inducing Ricci solitons, aside from the nearly parallel case where $\Delta\psi=\lambda^2\psi$ and $\Ric=\frac38\tau_0^2g$. For instance, in \cite{Moreno2019} the authors obtain an example of a Laplacian coflow soliton inducing a Ricci soliton on a solvable Lie group. 
   \end{itemize} 
\end{remark}

Using \eqref{Eq:Condition.soliton}, we can give an alternative proof for the non-existence of shrinking solitons in the compact case \cite{Karigiannis2012}*{Proposition 4.3}, and we extend this result to non-compact cases with $X$ divergence free:

\begin{cor}\label{cor: non-compact_shrinking_soliton}
\begin{enumerate}
    \item  There are no compact shrinking solitons of the Laplacian coflow.
    \item The only compact steady solitons of the Laplacian coflow are given by torsion free $\rG_2$-structures.
    \item There do not exist steady (non-trivial i.e. $X=0$) and shrinking solitons of the Laplacian coflow with $\div X=0$.
\end{enumerate}
\end{cor}

\begin{proof}
    Taking the trace on the second equation of \eqref{Eq:Condition.soliton}, we obtain
    \begin{equation}\label{eq: div X+ lambda}
        \frac{1}{2}\left((\tr T)^2+|T|^2\right)=\frac74 \lambda+\div X,
    \end{equation}
    since $\tr(T\circ T)=(\tr T)^2-|T|^2$ and $\tr \Ric = R$ (see \eqref{eq: Ricci_formula}). If $\div X=0$ then $\lambda\geq 0$. When the manifold $M$ is compact, we have
    $$
    \lambda\vol(M)=\frac27\int_M\left((\tr T)^2+|T|^2\right)\vol\geq 0.
    $$
    Hence, $\lambda > 0$ or $\lambda=0$ if and only if $T=0$. 
\end{proof}

%\begin{remark}
 % In Karigiannis \cite{Karigiannis2020} was shown that $\cL_X\varphi=0$ if and only if $\cL_X\psi=0$ if and only if $\cL_Xg=0$ and $\Curl X=X\lrcorner T$. Then, \textcolor{red}{when $X$ is a symmetric vector field, we obtain that the soliton solution of the Laplacian coflow \eqref{Eq. Soliton Laplacian coflow} is equivalent to soliton solution of the Laplacian flow.} In fact, suppose that $X$ is a symetric vector field and suppose that $(\varphi,\lambda)$ is a soliton solution of the Laplacian flow of $\gt$-structures, then we have that $\Delta_{\varphi}\varphi=\lambda\varphi$, Since $X$ is a symmetric vector field, then $\cL_X\psi=0$, which imply that $$\Delta_\psi\psi=\ast\Delta_\varphi\varphi=\ast\lambda\varphi=\lambda\psi$$
%\end{remark}
%Taking account the above remark, we could find symmetric vector field in $M$ which is a topological property and additionally observe when $\Curl X=X\lrcorner T$. 

%In particular, if $\cL_X\varphi=0$, we have that $X$ is a Killing vector field.  On the other hand, we say that a Killing vector field $X$ is nontrivial if it is not parallel. It is know that a nontrivial Killing vector field on a compact Riemannian manifold restricts its topology and geometry, for example, it does not allow the Riemannian manifold $(M,g)$ has a positive sectional curvature, then its fundamental group contains a cyclic subgroup with a \textcolor{red}{constant index}, depending only on $M$. 

\section{Almost Abelian Lie groups revisited}\label{Sec: almost abelian}
We study in this section the Laplacian coflow and its solitons in a class of solvable Lie group, which  has a codimension one Abelian ideal using the bracket flow as described in \cite{Lauret2016}.  

%\subsection{Almost Abelian Lie groups revisited}
Let $G$ be a Lie group, it is called \emph{almost Abelian} if its Lie algebra $\fg$ admits an Abelian ideal $\fh$ of codimension $1$.  For $\dim G=7$, any invariant $\gt$-structure is completely determined by a $\gt$-structure on $\fg$. Moreover, since $\gt$ acts transitively on the $6$-sphere, thus, for any orthonormal basis $\{e_1,...,e_7\}$, we can suppose that $e_7 \perp \fh$ and that the $\gt$-structure has the form  
\begin{equation}\label{coclosed_almost_abelian_G2}
\varphi=\omega\wedge e^7+\rho^+=e^{127}+e^{347}+e^{567}+e^{135}-e^{146}-e^{245}-e^{236},
\end{equation}
where $\omega=e^{12}+e^{34}+e^{56}$ and $\rho_+=e^{135}-e^{146}-e^{245}-e^{236}$ are the canonical $\SU(3)$--structure of $\mathfrak{h}\cong \R^6$. Additionally, the induced dual $4$-form is

\begin{equation}
\psi:=\ast\varphi=\frac{1}{2}\omega^2+\rho_-\wedge e^7=e^{1234}+e^{1256}+e^{3456}-e^{2467}+e^{2357}+e^{1457}+e^{1367},
\end{equation}
where $\rho_-=J^\ast\rho_+$ and $J$ is the canonical complex structure on $\mathbb{R}^6$ defined by $\omega:=\langle J\cdot,\cdot\rangle$. Moreover, the Lie bracket of $\mathfrak{g}$ is encoded by $A\in \mathfrak{gl}(\R^6)$ where $A:=\ad(e_7)|_{\mathfrak{h}}$

%Let $G$  be a $7$-dimensional connected and simply connected almost Abelian Lie group, in terms of its Lie algebra $\fg$  with an $G$-invariant $\gt$-structure $\varphi$, denote by $\mathfrak{g}$ the corresponding almost Abelian Lie algebra and  $\mathfrak{h}\simeq \R^6$ its codimension one Abelian ideal. Choose an orthonormal basis $\{e_1,...,e_7\}$ of $\mathfrak{g}$, with respect to the inner product 
%\begin{equation}\label{eq: canonical_inner_product}
%g_\varphi=\langle\cdot,\cdot\rangle=(e^1)^2+(e^2)^2+(e^3)^2+(e^4)^2+(e^5)^2+(e^6)^2+(e^7)^2,
%\end{equation}
%such that $\mathfrak{h}=\spanned\{e_1,...,e_6\}$, a $G$-invariant $\gt$--structure is determined by
%\begin{equation}\label{coclosed_almost_abelian_G2}
%\varphi=\omega\wedge e^7+\rho^+=e^{127}+e^{347}+e^{567}+e^{135}-e^{146}-e^{245}-e^{236},
%\end{equation}
%where
%$$
%\omega=e^{12}+e^{34}+e^{56} \qandq \rho_+=e^{135}-e^{146}-e^{245}-e^{236}
%$$
%are the canonical $\SU(3)$--structure of $\mathbb{C}^3\cong \mathfrak{h}$. Also, we have the natural dual $4$-form 
%\begin{equation}
%\psi:=\ast\varphi=\frac{1}{2}\omega^2+\rho_-\wedge e^7=e^{1234}+e^{1256}+e^{3456}-e^{2467}+e^{2357}+e^{1457}+e^{1367},
%\end{equation}
%where $\rho_-=J^\ast\rho_+$ and $J$ is the canonical almost structure on $\mathbb{R}^6$ defined by $\omega:=\langle J\cdot,\cdot\rangle$. In this sense, the Lie bracket of $\mathfrak{g}$ is  $A:=\ad(e_7)|_{\mathfrak{h}}$.

The transitive action of  $\GL(\fg)$ on the space of $\gt$-structures, defined by $h\cdot\varphi:=(h^{-1})^\ast\varphi$  ( for $h\in \GL(\fg)$), yields an infinitesimal representation of the alternating $3$-form
\begin{equation}\label{eq: theta_model_varphi}
\Lambda^3(\fg)^\ast=\theta(\mathfrak{gl}(\mathfrak{g}))\varphi 
\end{equation}
where $\theta: \mathfrak{gl}(\mathfrak{g})\rightarrow\End(\Lambda^3\mathfrak{g}^\ast)$ is defined by
\begin{equation}\label{eq: theta_on_varphi}
    \theta(B)\varphi:=\frac{d}{dt}\Big|_{t=0}e^{tB}\cdot\varphi=-\varphi(B\cdot,\cdot,\cdot)-\varphi(\cdot,B\cdot,\cdot)-\varphi(\cdot,\cdot,B\cdot).
\end{equation}
Since the orbit $\GL(\fg)\psi$ is also open in $\Lambda^4\fg^*$, the relation \eqref{eq: theta_model_varphi} also holds for the $4$-forms and $\psi$, namely $\Lambda^4(\fg)^\ast=\theta(\mathfrak{gl}(\mathfrak{g}))\psi$. 
%Denote by $\ast$ the Hodge star operator on $\Lambda^k\mathfrak{g}^\ast$ induced by $\varphi$ and $\star$ the induced Hodge star operator on $\Lambda^k\fh^\ast$. 
Coclosed $\gt$-structures on almost Abelian Lie algebras are equivalent with the Lie bracket constrain $A\in \mathfrak{sp}(\R^6)$ \cite{Freibert2012}, where $$
\mathfrak{sp}(\R^6)=\{A\in\mathfrak{gl}(\R^6): \quad AJ+JA^t=0  \quad \Leftrightarrow \quad \theta(A)\omega=0\}.$$

In particular, the non-vanishing torsion forms $\tau_0$ and $\tau_3$ can be described in terms of $A$:
\begin{prop}\cite{moreno2022}*{Prop. 3.2 \&  Cor 3.3}\label{torsion_forms}
	Let $(\mathfrak{g},A)$ be an almost Abelian Lie algebra with coclosed $\gt$-structure $\varphi$. Hence, the  torsion forms of $\varphi$ are
	\begin{align*}
	\tau_0=&\frac{2}{7}\tr(JA) \qandq 
	\tau_{27}=\left(\begin{array}{c|c}
	\frac{1}{14}\tr(JA)\rI_{6\times 6}-\frac{1}{2}[J,A] & 0 \\ \hline
	0 & -\frac{3}{7}\tr(JA)
	\end{array}
	\right).
    \end{align*}
    And its full torsion tensor is
	\begin{equation}\label{eq: torsion_tensor_A}
	T= \left(\begin{array}{c|c}
	\frac{1}{2}[J,A] & 0 \\ \hline
	0 & \frac{1}{2}\tr(JA)
	\end{array}
	\right).
    \end{equation}
    \end{prop}
%    For a coclosed $G_2$-structure, the divergence of the full torsion tensor and the Ricci curvature are \cite{Grigorian2013}*{Lemma 4.5 \& Eq (4.30)} 
	%\begin{align*}
	 %   \div T=d(\tr T) \qandq 
	%\Ric(g)=-\Curl(T)-T^2+(\tr T)T.
	%\end{align*}
Moreover, we can describe the Hodge Laplacian $\Delta \psi$ in function of $A\in \mathfrak{sp}(\R^6)$, hence according with Lemma \ref{Lemma. Laplacian Grigorian}, we first compute the tensor $T\circ T$ given in \eqref{eq:products}:

\begin{lemma}
	Let $(\fg,A)$ be an almost Abelian Lie algebra with coclosed $\gt$-structure $\varphi$. Denote by $S_A:=\frac{1}{2}(A+A^t)$ the symmetric part of $A$, then we have
	\begin{equation}\label{TcircT}
	T\circ T=\left(\begin{array}{c|c}
	-\frac{1}{2}(\tr JA)[J,A] - S_A\circ_6 S_A & 0 \\ \hline
	0 & -\tr S_A^2
	\end{array}\right),
	\end{equation}
	where $\circ_6$ is the product on $\mathfrak{gl}(\R^6)$, defined by $(S_A\circ_6S_A)_{ab}:=(S_A)_{mn}(S_A)_{pq}\rho^+_{mpa}\rho^+_{nqb}$.
\end{lemma}

\begin{proof}
%    By Proposition \ref{torsion_forms}, the full torsion tensor of $(\fg,\varphi)$, with $\varphi$ coclosed is 
 %   \begin{equation*}\label{torsion_sym}
  %      T=\frac{1}{2}\left(\begin{array}{c|c}
	%[J,A] & 0 \\ \hline
	%0 & \tr(JA)
	%\end{array}\right).
    %\end{equation*}
	We first compute the entry $(T\circ T)_{77}$, thus by \eqref{eq: torsion_tensor_A} we have 
	\begin{align*}
	(T\circ T)_{77}=&T_{mn}T_{pq}\varphi_{mp7}\varphi_{nq7}=\frac{1}{4}[J,A]_{mn}[J,A]_{pq}\omega_{mp}\omega_{nq}\\
	=&\frac{1}{4}(J[J,A])_{np}([J,A]J)_{np}\\
	=&\frac{1}{4}\langle J[J,A],[J,A]J\rangle=-\tr S_A^2,
	\end{align*}
	for the last equality we used $A=JA^tJ$ (i.e. $A\in \mathfrak{sp}(\R^6)$). Now,
for $i\neq 7$ and $j=7$, we have
	%\begin{equation*}
	%(T\circ T)_{i7}=T_{mn}T_{pq}\varphi_{mpi}\varphi_{nq7}=T_{mn}T_{pq}\varphi_{mpi}\omega_{nq}=\spadesuit.
	%\end{equation*}
	%Since $n,q\in\{1,...,6\}$ by equation \eqref{torsion_sym} also $m,p\in\{1,...,6\}$, then
	\begin{align*}
	(T\circ T)_{i7}=&T_{mn}T_{pq}\varphi_{mpi}\omega_{nq}
 =\frac{1}{4}[J,A]_{mn}[J,A]_{pq}\rho^+_{mpi}\omega_{nq}\\
 =&\frac{1}{4}([J,A]J[J,A])_{mp}\rho^+_{mpi}
 =\frac{1}{4}([J,A]^2)_{mn}J_{np}\rho^+_{pmi}\\
 =&-\frac{1}{4}([J,A]^2)_{mn}\rho^-_{nmi}=0.
	\end{align*}
For the above computation, we used that $[J,A]\in \mathfrak{sp}(\R^6)$ is symmetric and \eqref{SU3_identities}. 
	Finally, for $i\neq 7$ and $j\neq 7$ we have
	\begin{align*}
	(T\circ T)_{ij}
	=&2T_{mn}T_{77}\omega_{mi}\omega_{nj}+T_{mn}T_{pq}\rho^+_{mpi}\rho^+_{nqj}\\
	=&\frac{1}{2}(\tr JA) [J,A]_{mn}J_{mi}J_{nj}+\frac{1}{4}[J,A]_{mn}[J,A]_{pq}\rho^+_{mpi}\rho^+_{nqj}\\
	=&-\frac{1}{2}(\tr JA)(J[J,A]J)_{ij} +(JS_A)_{mn}(JS_A)_{pq}\rho^+_{mpi}\rho^+_{nqj}\\
	=&-\frac{1}{2}(\tr JA) ([J,A])_{ij}+J_{mk}(S_A)_{kn}J_{pl}(S_A)_{lq}\rho^+_{mpi}\rho^+_{nqj}\\
	=&-\frac{1}{2}(\tr JA) ([J,A])_{ij}-(S_A)_{kn}(S_A)_{lq}\rho^+_{kli}\rho^+_{nqj}.
	\end{align*}
 Once again, we used the identities \eqref{SU3_identities} Finally, combining each case of $i$ and $j$, we get the expression for $T\circ T$.
\end{proof}

Now, for almost 
Abelian Lie algebras $(\fg,A)$ the Ricci curvature is  \cite{Arroyo2013}*{Eq (8)}:
\begin{equation}\label{Ricci_operator}
	\Ric_A=\left(\begin{array}{c|c}
	\frac{1}{2}[A,A^t] & 0 \\ \hline
	0 & -\tr S_A^2
	\end{array}\right) \qforq A\in \mathfrak{sp}(\R^6),
	\end{equation}
 and since $\tr(T)$ is constant, we have $\div T=0$. Therefore, we can write Lemma \ref{Lemma. Laplacian Grigorian} in function of the Lie bracket:
 
\begin{prop}\label{coflow_symbol}
	Let $(\fg,A)$ be an almost Abelian Lie algebra with coclosed $\gt$-structure $\varphi$. Thus, the Hodge Laplacian of $\psi$ is $\Delta_A\psi=\theta(Q_A)\psi$ where
		
		\begin{align}\label{homogeneous_Q}
		\begin{split}
		Q_A=&\left(\begin{array}{c|c}
		Q_A^\fh & 0 \\ \hline
		0   & q_A
		\end{array}\right)		=\left(\begin{array}{c|c}
		\frac{1}{2}[A,A^t]+\frac{1}{2}S_A\circ_6S_A & 0 \\ \hline
		0   & -\frac{1}{2}\tr(S_A)^2-\frac{1}{4}(\tr JA)^2
		\end{array}\right)
		\end{split}
		\end{align}
		%\item[ii)] For the modified Laplacian 
		%$$
		%\Delta_\psi\psi+2d\big((C-\tr T)\varphi\big)=\theta\Big(\ricci(g)-\frac{1}{2}T\circ T-(2C-\tr T)T\Big)=\theta(P_A)
		%$$
		%where 
		%\begin{equation}\label{Q_A_symbol}
		%P_A=\left(\begin{array}{c|c}
		%P_1 & 0 \\ \hline
		%0   & p
		%\end{array}\right),
		%\end{equation}
		%where $P_1=\frac{1}{2}[A,A^t]+\frac{1}{2}S_A\circ_6S_A-\bigg(C-\frac{1}{2}\tr JA\bigg)[J,A]$ and $p=-\frac{1}{2}\tr(S_A)^2+\frac{1}{4}(\tr JA)^2-C\tr JA$.
		In particular, $Q_A\in \mathfrak{gl}(\mathfrak{g})$ is symmetric.
\end{prop}
\begin{proof}
The result follows by applying equations \eqref{eq: torsion_tensor_A}, \eqref{Ricci_operator} and  \eqref{TcircT} into Lemma \ref{Lemma. Laplacian Grigorian}.
\end{proof}

Notice that the closed condition on $\psi$ implies that $\Delta_\psi\psi=dd^\ast\psi$ is also closed.  Similarly, it is interpreted as $Q_A^\fh\in\mathfrak{sp}(\R^6)$ for $A\in\mathfrak{sp}(\R^6)$. Indeed:

\begin{lemma}\label{S_product_S}
	If $A\in \mathfrak{sp}(\R^6)$ is symmetric then $A\circ_6 A\in \mathfrak{sp}(\R^6)$.
\end{lemma}
\begin{proof}
Notice that $B:=A\circ_6 A$ is symmetric, thus it is enough to prove the equality $JBJ=B$. Hence
	\begin{align*}
	(JBJ)_{ij}=&J_{ik}(A\circ_6A)_{kl}J_{lj}=J_{ik}A_{mn}A_{pq}\rho^+_{mpk}\rho^+_{nql}J_{lj}\\
    =&-(JAJ)_{mn}A_{pq}\rho^-_{mpi}\rho^-_{nqj}=-J_{mr}A_{rs}J_{sn}A_{pq}\rho^-_{mpi}\rho^-_{nqj}\\
    =&A_{rs}A_{pq}\rho^+_{rpi}\rho^+_{sqj}=B_{ij}
    \end{align*}
	Later, we used the identities \eqref{SU3_identities} time and again, as well as the symmetry of $A$.     
\end{proof}

\subsection{The bracket flow} 
In this section we adapt the general approach of geometric flows of homogeneous geometric structures,  proposed by J. Lauret, to the framework of the Laplacian coflow \eqref{Eq.coflowLaplacian} on almost Abelian Lie algebras with coclosed $\gt$-structures, for a broad exposition see  \cite{Lauret2016}.

Let $\{\varphi(t)\}_{t\in(\epsilon_1,\epsilon_2)}$ be a solution of the Laplacian coflow on $(\fg,A)$ with initial condition $\varphi(0)=\varphi_0$. Since $\varphi(t)\in \GL(\fg)\varphi_0$, we can write $\varphi(t)=h(t)^*\varphi_0$ for $h(t)\in \GL(\fg)$ satisfying $h(0)=I$. Since $\ast_{\varphi(t)}=(h^{-1})^**_{\varphi_0}h^*$ (see \cite{Moreno2019}*{Lemma 3.1}), we can write $\psi(t)=h(t)^*\psi_0$ for $\psi_0=\ast_{\varphi_0}\varphi_0$ and  by Proposition \ref{coflow_symbol}, we have $$\Delta_A\psi(t)=\theta(Q_A(t))\psi(t),$$
hence, the Laplacian coflow is equivalent with
%Summarily, if the solution of \eqref{co-flow} is restricted to the $G$-invariant ones, the PDE becomes an ODE on $\Lambda^3_+\fg^\ast$. For an initial condition  $\varphi(0)=\varphi_0$, the solution $\varphi(t)$ lies in the orbit $\GL(\fg)\varphi_0$. Thus, there is a one parameter subgroup $h(t)\in\GL(\fg)$, such that 
%$$
%\varphi(t)=h(t)^\ast\varphi_0=h(t)^{-1}\cdot\varphi_0 \qandq \psi(t)=h(t)^\ast\psi_0=h(t)^{-1}\cdot\psi_0.$$ 
%Under the flow \eqref{co-flow}, the matrix $h(t)$ evolves by
%\begin{equation}\label{endomorphism_flow}
 %   \frac{d}{dt}h_t=-h_tQ_{A(t)}+2k(C-\tr T)h_tT_t,
%\end{equation}
\begin{equation}\label{endomorphism_flow}
    \frac{d}{dt}h(t)=-h(t)Q_{A}(t)%+2k(C-\tr T)h_tT_t,
\end{equation}
%for $k=0$ and $k=1$, respectively.

\begin{definition}
    Let $(G_1,\varphi_1)$ and $(G_2,\varphi_2)$ be Lie groups with $\gt$-structure $\varphi_i$ (for $i=1,2$). An isomorphism $f:(G_1,\varphi_1)\rightarrow (G_2,\varphi_2)$ is called an \emph{equivariant isomorphism} if it is a Lie group isomorphism such that $\varphi_1=f^*\varphi_2$ and this case, $(G_1,\varphi_1)$ and $(G_2,\varphi_2)$ are called \emph{equivariant equivalent}.
\end{definition}

 Since $\varphi(t)=\ast_t\psi(t)$ induces a $\SU(3)$-structure on $\fh$ for each $t$, we can write 
\begin{equation}\label{eq: h_endomorphism}
    h(t)=k(t)+a(t)e^7\otimes e_7 \qwhereq k(t)\in \Gl(\R^6) \qandq a(t)\in \R^*.
\end{equation}
Thus, we can defined the bracket $A(t)=a(t)^{-1}k(t)Ak(t)^{-1}$ such that \eqref{eq: h_endomorphism} becomes a Lie algebra isomorphism between  $(\fg,A,\varphi(t))$ and $(\fg,A(t),\varphi)$ such that $\varphi(t)=h(t)^*\varphi$. Moreover, since $\Delta_A\psi(t)=h(t)^\ast\Delta_{A(t)}\psi$ we get the relation $Q_{A(t)}=h(t)Q_A(t)h(t)^{-1}$ and consequently, the equation \eqref{endomorphism_flow} becomes an ODE on $(\fg,A(t),\varphi)$ 
\begin{align}\label{eq: evolution of h with A(t)}
    \frac{d}{dt}h(t)=-Q_{A(t)}h(t)
\end{align}
In particular, under the flow \eqref{eq: evolution of h with A(t)} the bracket $A(t)$ evolves by:
\begin{equation}\label{bracket_flow}
    \frac{d}{dt}A(t)=q_{A(t)}A(t)-[Q_{A(t)}^\fh,A(t)],
\end{equation}
where $q(t)$ and $Q_{A(t)}^\fh$ are defined in \eqref{homogeneous_Q} for each $t\in (\epsilon_1,\epsilon_2)$.
%Define the Lie bracket  $\mu(t)(\cdot,\cdot)=h(t)[h(t)^{-1}\cdot,h(t)^{-1}\cdot]$ on $\mathfrak{g}$, it induces a Lie algebra isomorphism 
%$$
%h(t):(\fg,[\cdot,\cdot],\varphi(t))\rightarrow (\fg,\mu(t),\varphi_0),
%$$
%such that $\varphi(t)=h(t)^\ast\varphi_0$. In particular, we get an equivalence $f(t): (G,\varphi(t))\rightarrow (G_{\mu(t)},\varphi_0)$ satisfying $\varphi(t)=f(t)^\ast\varphi_0$ and $(df(t))_{1_G}=h(t)$, where $G$ and $G_{\mu(t)}$ are the $1$-connected Lie groups with Lie algebras $(\fg,[\cdot,\cdot])$ and $(\fg,\mu(t))$, respectively. Under the flow \eqref{endomorphism_flow} the bracket $\mu(t)$ evolves by
%\begin{equation}\label{bracket_flow}
 %   \frac{d}{dt}\mu_t=\delta_{\mu_t}( Q_{\mu_t}+2k(C-\tr T)T_{\mu_t})
%\end{equation}
%\begin{equation}\label{bracket_flow}
 %   \frac{d}{dt}\mu(t)=\delta_{\mu(t)}( Q_{\mu(t)})%+2k(C-\tr T)T_{\mu_t})
%\end{equation}
%where $Q_{\mu(t)}=\Ad(h(t))Q_{A(t)}$, $T_{\mu(t)}=\Ad(h(t))T_t$ and 
%$$
 % \delta_{\mu}(E)=\mu(E\cdot,\cdot)+\mu(\cdot,E\cdot)-E\mu(\cdot,\cdot), \quad \forall E\in \mathfrak{gl}(7,\R).
%$$
The ODE \eqref{bracket_flow} is known as the \emph{bracket flow} and it provides an equivalent analysis of the geometric flow of homogeneous geometric structures, varying the Lie bracket instead of the geometric structure:

\begin{theorem}[\cite{Lauret2016}*{Theorem 5}]\label{eq: Lauret_theorem 1}
Let $\{\varphi(t)\}_{t\in(\epsilon_1,\epsilon_2)}$ be a solution of the Laplacian coflow on $(\fg,A)$ with initial condition $\varphi(0)=\varphi_0$. Then, there exist an equivariant isomorphism $f(t): (G_A,\varphi(t))\rightarrow (G_{A(t)},\varphi)$, such that $h(t)=df(t)_1$ solves either \eqref{endomorphism_flow} or \eqref{eq: evolution of h with A(t)} for all $t\in (\epsilon_1,\epsilon_2)$. %and $h(t)=(df(t))_{1_G}:\fg\rightarrow \fg$ is the solution of the following ODE's:
%    \begin{itemize}
 %       \item[(i)] $\frac{d}{dt}h(t)=-h(t)Q_{A(t)}$, $h(0)=\rI_7$, where $Q_{A(t)}\in\fq$ satisfies $\Delta_{\psi(t)}\psi(t)=\theta(Q_{A(t)})\psi(t)$.
  %      \item[(ii)] $\frac{d}{dt}h(t)=-Q_{\mu(t)}h(t)$, $h(0)=\rI_7$, where $Q_{\mu(t)}\in\fq$ satisfies $\Delta_{\mu(t)}\psi=\theta(Q_{\mu(t)})\psi$.
   % \end{itemize}
    In addition, the solutions of \eqref{Eq.coflowLaplacian} and \eqref{bracket_flow} are
    $$
      \varphi(t)=h(t)^\ast\varphi \qandq A(t)=a(t)k(t)^{-1}Ak(t),
    $$
    respectively,  for $t\in(\epsilon_1,\epsilon_2)$.
\end{theorem}

 Theorem \ref{eq: Lauret_theorem 1} provides a useful tool for addressing long-time existence and regularity questions, since it shows that the Laplacian coflow and the bracket flow have the same maximal interval of solution. Hence, the bracket flow \eqref{bracket_flow} is explicitly given by 

\begin{prop}
    Let $\mathcal{L}\simeq \fg\fl(\R^6)$ be the family of $7$-dimensional almost Abelian Lie algebras. The subfamily $\cL_{coclosed} \simeq\mathfrak{sp}(\mathbb{R}^6)\subset \mathcal{L}$ of coclosed $\gt$-structures is invariant under the bracket flow \eqref{bracket_flow}, which becomes equivalent to the following ODE for a one-parameter family of matrices $A=A(t)\in \mathfrak{sp}(\mathbb{R}^6)$:
	\begin{align}
	\begin{split}\label{matrix_ODE}
	\frac{d}{dt}A=&-\Big(\frac{1}{2}\tr(S_A)^2+\frac{1}{4}(\tr JA)^2\Big)A+\frac{1}{2}[A,[A,A^t]]+\frac{1}{2}[A,S_A\circ_6S_A]
	\end{split}
	\end{align}
\end{prop}

\begin{proof}
	Notice that the velocity $\dot{A}(t)=q_{A(t)}A+[A,Q_{A(t)}^\fh]$ lies in $\mathfrak{sp}(\mathbb{R}^6)$, since $S_A\circ_6S_A\in\mathfrak{sp}(\mathbb{R}^6)$ by Lemma \ref{S_product_S}, hence, the family $\cL_{coclosed}\subset \mathcal{L}$ is invariant under the bracket flow. Finally, replacing \eqref{homogeneous_Q} into \eqref{bracket_flow},  we obtain \eqref{matrix_ODE}.%is easy to prove that $\delta_\mu(Q_A)(e_i,e_j)=0$ for $e_i,e_j\in \fh$, also we have
%	\begin{align*}
	%\delta_\mu(Q_A)(e_7,e_i)=&\mu_A(Q_Ae_7,e_i)+\mu_A(e_7,Q_Ae_i)-Q_A\mu_A(e_7,e_i)\\  
	%=&q\mu_A(e_7,e_i)+\mu_A(e_7,Q_1e_i)-Q_1\mu_A(e_7,e_i)\\
	%=&(qA+AQ_1-Q_1A)e_i.
	%\end{align*}
	%Hence,  Therefore, the subset of invariant coclosed $\gt$-structures is invariant under the bracket flow and the bracket flow evolves by $\dot{A}=B$.
	\end{proof}

\begin{prop}\label{evolution_of_the_norm}
	If $A(t)$ is a bracket flow solution associated to the Laplacian coflow, then its norm evolves by
	\begin{align}\label{evolution_of_the_norm2}
	\frac{d}{dt}|A|^2=&-\Big(|S_A|^2+\frac{1}{2}(\tr JA)^2\Big)|A|^2-|[A,A^t]|^2-\langle S_A\circ_6S_A,[A,A^t]\rangle
	\end{align}
\end{prop}
\begin{proof}
	From equation \eqref{matrix_ODE}, we have
	\begin{align*}
	\frac{d}{dt}|A|^2=&2\langle \dot{A},A\rangle=2\tr(\dot{A}A^t)\\
	=&-\Big(\tr(S_A)^2+\frac{1}{2}(\tr JA)^2\Big)|A|^2 +\tr([A,[A,A^t]]A^t)+\tr([A,S_A\circ_6S_A],A^t)\\
	=&-\Big(|S_A|^2+\frac{1}{2}(\tr JA)^2\Big)|A|^2-|[A,A^t]|-\langle S_A\circ_6S_A,[A,A^t]\rangle
	\end{align*}
\end{proof}	

In order to proof long time existence solution for \eqref{matrix_ODE} we need the following identity.

\begin{lemma}\label{S_circ_S}
	For the symmetric part $S_A$ of the matrix $A\in\mathfrak{sp}(\mathbb{R}^6)$, we have
	$$
	|S_A\circ_6 S_A|^2=4(|S_A|^2|S_A|^2-2|S_A^2|^2-\langle JS_A,S_A\rangle^2).
	$$
\end{lemma}

\begin{proof}
	This identity follows by direct computations, using the contractions \eqref{SU3_identities} and \eqref{eq: rho+rho+}.
\end{proof}

\begin{theorem}
	The Laplacian coflow solution $(\fg,A,\varphi(t))$ starting at any coclosed (non-flat) $\gt$-structure is defined for all $t\in (\epsilon_1,\infty)$.
\end{theorem}

\begin{proof}
Let $\varphi(t)$ a solution of the Laplacian coflow defined for all $t\in(\epsilon_1,\epsilon_2)$, according to Theorem \ref{eq: Lauret_theorem 1}, we get that the solution $A(t)\in \mathfrak{sp}(\R^6)$ of \eqref{matrix_ODE} is defined for all $t\in (\varepsilon_1,\varepsilon_2)$. %Now, we are going to study the term 
%$$
%\tr (S_A\circ_6 S_A[A,A^t])=\langle S_A\circ_6S_A,[A,A^t]\rangle
%$$
%given in \eqref{evolution_of_the_norm2}. 
 Now, using the Cauchy-Schwarz and Peter-Paul inequalities (i.e. $ab\leq \frac{a^2}{4}+ b^2$ for $a,b\geq 0$), we have
\begin{align}\label{CSPP-inequality}
\begin{split}
-\langle S_A\circ_6S_A,[A,A^t]\rangle\leq& |S_A\circ_6S_A||[A,A^t]|\\
\leq&  \frac{|S_A\circ_6S_A|^2}{4}+|[A,A^t]|^2= |S_A|^2|S_A|^2-2|S_A^2|^2-\langle JS_A,S_A\rangle^2+|[A,A^t]|^2.
\end{split}
\end{align}
Replacing the last inequality into equation \eqref{evolution_of_the_norm2}, we have
\begin{align*}
\frac{d}{dt}{|A|^2}\leq&-\Big(|S_A|^2+\frac{1}{2}(\tr JA)^2\Big)|A|^2-|[A,A^t]|^2+|S_A|^2|S_A|^2-2|S_A^2|^2-\langle JS_A,S_A\rangle^2+|[A,A^t]|^2\\
=&-|S_A|^2|S_A|^2-\frac{1}{4}|S_A|^2|A-A^t|^2-\frac{1}{2}(\tr JA)^2|A|^2+|S_A|^2|S_A|^2-2|S_A^2|^2-\langle JS_A,S_A\rangle^2\\
=&-\frac{1}{4}|S_A|^2|A-A^t|^2-\frac{1}{2}(\tr JA)^2|A|^2-2|S_A^2|^2-\langle JS_A,S_A\rangle^2\leq 0
\end{align*}
Thus, $|A|^2$ is non-increasing and non-negative, therefore $A(t)$ is an inmortal solution, i.e. it is defined for all $t\in(\varepsilon_1,\infty)$. In particular, $|A|^2$ is  strictly decreasing unless $(\fg,A(t),\varphi)$ is torsion free, that is 
$$
\dot{|A|^2}=0 \quad \Leftrightarrow \quad A^t=-A \qandq \tr JA=0
$$
and thus $A(t)\equiv A_0\in \mathfrak{sl}(\C^3)\cap \mathfrak{sp}(\R^6)=\su(3)$ the bracket flow solution is constant.
\end{proof}

\begin{remark}
	 In \cite{Fino2018} Bagaglini and Fino address also the Laplacian coflow on almost Abelian Lie algebras, there the approach is different  from ours, the authors find explicit solutions of the Laplacian coflow when $A\in \mathfrak{sp}(\mathbb{R}^6)$ is normal.  Notice that the above Corollary holds for any $A\in  \mathfrak{sp}(\mathbb{R}^6)$. 
\end{remark}
	
\begin{prop}\label{Prop. Example}
  Let $(\mathfrak{g},A)$ be an almost Abelian algebra with the matrix $A$ defined by 
  \begin{equation}\label{Eq. Matrix.EDO}
	    A=\left[
	    \begin{array}{c|c}
	    B & 0 \\ \hline
	    0 & -B^t
	    \end{array}
	    \right] \qwithq B=\left[\begin{array}{ccc}
	    0 & x & 0 \\ 
	    y & 0 & 0 \\
	    0 & 0 & 0
	    \end{array}
	    \right] \qandq x,y\in\R
	\end{equation}
 Then, the bracket flow $A(t)$ given by \eqref{matrix_ODE} is stable. 
\end{prop}  
\begin{proof}
    The bracket flow \eqref{matrix_ODE} being $A(t)$ the matrix given by \eqref{Eq. Matrix.EDO} can be analized by the following nonlinear system $\dot{\bf{x}}=\bf{f}(\bf{x})$ where $\bf{f}: \mathbb{R}^2\rightarrow\mathbb{R}^2$; $\bf{x}\mapsto \bf{f}=(f_1(\bf{x}),f_2(\bf{x}))$ given by
	\begin{equation}\label{Eq:system1}
	 \dot{x}=-2x(3x-y)(x+y) \qandq \dot{y}=2y(x-3y)(x+y).
	\end{equation}
A point $\bx\in \R^2$ is an \emph{equlibrium point} if $\bf{f} (x)=0$ which is given by the surface $S=\{(x,y)\in\R^2: x=-y\}$. The $x$-nullclines (i.e, $\bx\in\R^2$ where $\bf{f}_1(\bx)=0$) are the lines $x=0$, $x=-y$ and $y=3x$ 
and the $y$-nullclines (i.e, $\bx\in \R^2$ where $\bf{f}_2(x)=0$) are the lines $y=0$, $x=-y$ and $y=\frac{1}{3}x$. The intersection of the $x$-nulcines and $y$-nulclines yield the equilibrium points.
\begin{figure}
    \centering
    \includegraphics[width=0.5\textwidth]{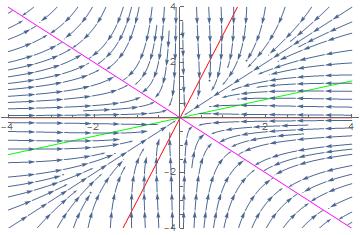}
    \caption{\textcolor{green}{x-nulcline}, \textcolor{red}{y-nulcline}, \textcolor{magenta}{equilibrium points}}
    \label{fig:my_label}
\end{figure}
On the other hand, the lines $y=0$ and $x=0$ are invariants for the system \eqref{Eq:system1}. If we set $y=0$ then we obtain $\dot{x}=-6x^3$. Therefore, $\dot{x}$ is positive if $x>0$ and negative if $x<0$ which clearly shows the stability along the line $y=0$.  

To determine the trajectories,  if $\bf{x_0}$ is not a equilibrium point  then, at least one of $f_1(\bf{x}_0)$ or $f_2(\bf{x}_0)$ is not zero. Let us suppose that $f_1(\bf{x}_0)\neq 0$. Then, there is an open neighborhood of $\bf{x}_0$, such that $f_1(\bf{x}_0)\neq 0$, so the orbit through $\bf{x}_0$ can be defined as a solution of the nonautonomous scalar equation 
\begin{equation}
    \frac{dy}{dx}=-\frac{y(x-3y)}{x(3x-y)}.
\end{equation}
This differential equation is homogeneous. Setting $y=xv(x)$, we obtain
\begin{align*}
    v+x\frac{dv}{dx}=&-\frac{v(1-3v)}{3-v}.
\end{align*}
    That is, 
\begin{equation*}
    x\frac{dv}{dx}=4v\left(\frac{v-1}{3-v}\right).
\end{equation*}
The resulting ODE is separable, with solution $x^{-4}v^{-3}(v-1)^2=c$. Reverting back to the original variables, the trajectories are level curves of 
\begin{equation*}
    H(x(t),y(t))=\frac{(y(t)-x(t))^2}{y(t)^3x(t)^3}.
\end{equation*}
%and it is qualitatively determined by the behaviour of the linear system  
%$$\dot{X}=DF(\Tilde{X}) $$
%So, the general linearisation for the system \eqref{Eq:system1} is given by
%\begin{equation*}
%DF(a,b) = 
%\begin{pmatrix}
%2 (-9 a^2 - 4 ab + b^2) &  4 a (-a + b)\\
% 4 b (a - b) & 2 (a^2 - 4 a b - 9 b^2)
%\end{pmatrix}
%\end{equation*} 
On the other hand, let
$$ V(\bx)=x^2+y^2+2xy,$$
be a Lyapunov function. In fact, $V(\bx)=0$ when $\bx$ is an equilibrium point for this system and $V({x})=(x+y)^2\geq 0$ if $\bf{x}$ is not an equilibrium point. Computing $\dot{V}(\bx)$, we find 
\begin{equation}
   \dot{V}(\bx)=-2(x+y)^2(6x^2-4xy+6y^2), 
\end{equation}
where $\dot{V}(x)=0$ if $x=-y$ and $\dot{V}(x)\leq 0$ otherwise. For any curve $\gamma(r,\theta)=(r\cos\theta, r\sin\theta)$ with $r>0$ and $0\leq\theta\leq 2\pi$, we obtain 
\begin{equation}
    \dot{V}(\gamma(r,\theta))=-2r^2(r\sin\theta+r\cos\theta)^2(6-2\sin(2\theta))\leq 0,
\end{equation}
since $|\sin(2\theta)|\leq 1$ then we have $6-2\sin(\theta)>0$.  Therefore, the system is stable if $\bx_0$ is a equilibrium point. 
\end{proof}

%$A\rightarrow A_0$ where $A_0\in \mathfrak{su}(3)$ and $A\rightarrow 0$ when $x=y$. 
 
 %The solutions in the first and third quadrant converge to zero and approach to the soliton solution given by the lines $x=y$, $x=0$ or $y=0$.
	%	\item The solutions in the second and fourth quadrant converge to a constant matrix given in the line $b=-a$.

%  Similar to the Laplacian flow of closed $\gt$-structures, the bracket solution $A(t)$ of \eqref{matrix_ODE} does not necessarily converge to zero.   Consider 

%	\begin{equation*}
%	    A=\left[
%	    \begin{array}{c|c}
%	    B & 0 \\ \hline
%	    0 & -B^t
%	    \end{array}
%	    \right] \qwithq B=\left[\begin{array}{ccc}
%	    0 & x & 0 \\ 
%	    y & 0 & 0 \\
%	    0 & 0 & 0
%	    \end{array}
%	    \right] \qandq x,y\in\R
%	\end{equation*}

\begin{prop}\label{gap_scalar_curvature}
	The scalar curvature $R(t)$ of an invariant Laplacian coflow solution starting at a non-flat coclosed $\gt$-structure is strictly increasing and satisfies the inequality
 \begin{equation}\label{eq: scalar_inequality}
	\frac{1}{-\frac{|t|}{2}+\frac{1}{R(0)}}\leq R(t)\leq 0 \quad \text{for} \quad \text{any} \quad t\in (\varepsilon_1,\infty),
	\end{equation}
	%\begin{equation}\label{eq: scalar_inequality}
	%\frac{1}{-\frac{t}{2}+\frac{1}{R(0)}}\leq R(t)\leq 0 \quad \text{for} \quad \text{any} \quad t\in [0,\infty),
	%\end{equation}
 %and 
 %\begin{equation}
  %      \frac{1}{\frac{t}{2}+\frac{1}{R(0)}}\leq R(t)<0 \quad \text{any} \quad t\in (T_-,0].
   % \end{equation}
	In particular, $|T|^2$ is strictly decreasing.  and it converges to zero when $t\rightarrow \infty$ as $|A(t)|^2\rightarrow 0$.
\end{prop}

\begin{proof}
	From \eqref{Ricci_operator}, we have $R=-\tr S_A^2=-\frac{1}{4}\tr(A+A^t)^2$. Thus, using the bracket flow equation \eqref{matrix_ODE} we have
	\begin{align*}
	\frac{d}{dt}\tr (A+A^t)^2=& 2\tr\Big((A+A^t)\frac{d}{dt}(A+A^t)\Big)\\
	=& - \Big(\tr S_A^2+\frac{1}{2}(\tr JA)^2\Big)\tr(A+A^t)^2+\tr\Big((A+A^t)[A-A^t,[A,A^t]]\Big)\\
	&+\tr\Big((A+A^t)[A-A^t,S_A\circ_6 S_A]\Big)\\
	=& - 4\Big(\tr S_A^2+\frac{1}{2}(\tr JA)^2\Big)\tr S_A^2+\tr\Big([A+A^t,A-A^t][A,A^t]\Big)\\
	&+\tr\Big([A+A^t,A-A^t]S_A\circ_6 S_A\Big)\\
	=& - 2\Big( |S_A|^2+(\tr JA)^2\Big)| S_A|^2-2|[A,A^t]|^2-2\langle[A,A^t]S_A\circ S_A\rangle
	\end{align*}
	Using the inequality \eqref{CSPP-inequality}, we obtain
	\begin{align*}
	\frac{d}{dt}\tr (A+A^t)^2\leq& - 2\Big(2| S_A|^2+(\tr JA)^2\Big)| S_A|^2+2| S_A|^4-4|S_A^2|^2-2(\langle  JS_A,S_A\rangle)^2\\
	\leq& -2(\tr S_A^2)^2=-\frac{1}{8}(\tr(A+A^t)^2)^2.
	\end{align*}
	For any $t_1,t_2\in (\varepsilon_1,\infty)$ satisfying $t_1\leq t_2$, the last inequality implies
    \begin{equation*}
        \frac{1}{R(t_2)}-\frac{1}{R(t_1)}\geq \frac{t_2-t_1}{2}.
    \end{equation*}
    If $t_1=0$ then we get
    \begin{equation*}
        \frac{1}{-\frac{t_2}{2}+\frac{1}{R(0)}}\leq R(t_2)<0 \quad \text{any} \quad t_2\in [0,\infty).
    \end{equation*}
    If $t_2=0$ then we obtain
    \begin{equation*}
        \frac{1}{\frac{t_1}{2}+\frac{1}{R(0)}}\leq R(t_1)<0 \quad \text{any} \quad t_1\in (\varepsilon_1,0].
    \end{equation*}
	%$$
	%\tr(A+A^t)^2\leq \frac{1}{\frac{t}{8}+\frac{1}{\tr(A(0)+A(0)^t)^2}} \quad \text{for} \quad \text{any} \quad t\in [0,\infty)
	%$$
	 Finally, by \eqref{eq: Laplacian.Grigorian} and  \eqref{eq: torsion_tensor_A}, the scalar curvature of a coclosed $\gt$-structure is 
  
  $$
	  R_A=-|T|^2+(\tr(JA))^2.
	$$
	Hence, using the Cauchy-Schwarz inequality, we have
	\begin{align*}
	    |T|^2\leq -R(t)+|J|^2|A(t)|^2=-R(t)+6|A(t)|^2\leq \frac{1}{\frac{|t|}{2}-\frac{1}{R(0)}}+6|A(t)|^2.
	\end{align*}
	Therefore, $|T|^2$ is strictly decreasing, since $|A(t)|^2$ is strictly decreasing as well and $|T|^2$ goes to zero as $|A(t)|\rightarrow 0$.
\end{proof}

\subsection{Algebraic solitons} 
In this section, we  characterise the invariant $\gt$-structures on almost Abelian Lie algebras which are solitons of the Laplacian coflow, in terms of the Lie bracket $A\in \mathfrak{sp}(\R^6)$.

%A solution $\psi(t)$ of flow \eqref{co-flow} is called \emph{self-similar} if 
%\begin{equation}\label{eq: self_similar_solution}
 %   \psi(t)=\lambda(t)f(t)^\ast\psi \qwithq \lambda(t)\in\R^\ast \qandq f(t)\in\Diff(G).
%\end{equation}
%For the initial condition $\psi(0)=\psi$, i.e. $c(0)=1$ and $f(0)=\mathrm{Id}$, the existence of a solution as \eqref{eq: self_similar_solution} is equivalent with a time independent equation
%\begin{equation}\label{eq: soliton equation}
 %   \Delta_\psi\psi=\lambda\psi+\cL_X\psi \qforq \lambda\in\R \qandq X\in \fX(G) \quad (X \quad \text{complete})
%\end{equation}
%\begin{definition}
%A $\gt$-structure $\varphi$, whose dual $4$-form $\psi$ is a \emph{soliton of the Laplacian coflow} if it satisfies the \emph{soliton equation} \eqref{eq: soliton equation}. In addition, if the constant is $\lambda<0, \lambda=0$ or $\lambda>0$, then $\psi$ is called a \emph{shrinking}, \emph{steady} or \emph{expanding} soliton, respectively.  
%\end{definition} 
The self-similar solution \eqref{eq: self-similar_solution} on the almost Abelian Lie group $G_A$ is invariant if $\lambda(t)\in \R^\ast$ and
$f(t)\in \Aut(G_A)$, then \eqref{eq: self-similar_solution} corresponds with the solution
\begin{equation}\label{eq: invariant_self-similar_solution}
    \psi(t)=\lambda(t)h(t)^\ast\psi\in \Lambda^4\fg^\ast \qwithq \lambda(t)\in\R^\ast \qandq h(t)\in\Aut(\fg,A),
\end{equation}
with $df(t)_{1}=h(t)$ and then, the soliton equation \eqref{Eq.Soliton} becomes $\Delta_\psi\psi=\lambda\psi+\cL_{X_D}\psi\in\Lambda^4\fg^\ast$ with $\lambda\in\R$ and $X_D:=\frac{d}{dt}|_{t=0}h(t)=:-D\in\der(\fg,A)$. Using the representation \eqref{eq: theta_model_varphi}, we have
\begin{align*}
    \theta(Q_A)\psi=&\Delta_\psi\psi=\lambda\psi+\cL_{X_D}\psi\\
    =&\theta\left(-\frac{\lambda}{4}I_7\right)\psi+\frac{d}{dt}|_{t=0}h(t)^\ast\psi\\
    =&\theta\left(-\frac{\lambda}{4}I_7+D\right)\psi.
\end{align*}
By Proposition \ref{coflow_symbol}, the matrix $Q_A$ is symmetric, hence, setting $\lambda=-4c$ we say that $\psi$ is a \emph{semi-algebraic} soliton if 
\begin{equation*}
    Q_A=cI_7+\frac{1}{2}(D+D^t),
\end{equation*}
and $\psi$ is an \emph{algebraic soliton} if $D^t\in \der(\fg,A)$. Moreover, the self-similar solution \eqref{eq: invariant_self-similar_solution} is given
\begin{equation*}
    \lambda(t)=(1-2ct)^2 \qandq h(t)=e^{-s(t)D} \qwhereq s(t)=-\frac{1}{2c}\log|2ct-1|
\end{equation*}
(For $c=0$ set $s(t)=t$). And the corresponding bracket solution of a semi-algebraic soliton is 
\begin{align}\label{eq: Auto_similar_sol_bracket}
    A(t)=(1-2ct)^{-1/2}e^{s(t)E}Ae^{-s(t)E} \qwhereq E=\frac{1}{2}(D-D^t),
\end{align}
(e.g. \cite{lafuente2014}*{Remark 3.4} for the homogeneous Ricci soliton case). 
The next Theorem shows the (semi-) algebraic soliton equation in terms $A\in \mathfrak{sp}(\mathbb{R}^6)$.

\begin{theorem}\label{alge_soliton_prop}
	Let $(\fg,A,\varphi)$ be an almost Abelian Lie algebra with coclosed $\gt$-structure:
	\begin{itemize}
	    \item[(i)] $\psi$ 	is an algebraic soliton for the Laplacian coflow if and only if 
	\begin{equation}\label{eq: algebraic_soliton-condition}
		[[A,A^t]+S_A\circ_6 S_A,A]=\frac{|[A,A^t]|^2+\langle S_A\circ_6 S_A,[A,A^t]\rangle}{|A|^2}A
	\end{equation}
	In this case, $D=Q_A-cI_7\in \der(\fg, A)$ for 
	\begin{equation}\label{constant_c}
		c=-\frac{1}{2}\Big(\tr S_A^2+\frac{1}{2}(\tr JA)^2+\frac{|[A,A^t]|^2}{|A|^2}+\frac{\langle S_A\circ_6 S_A,[A,A^t]\rangle}{|A|^2}\Big).
	\end{equation}
	 \item[(ii)] $\psi$ is a semi-algebraic soliton if and only if
	 \begin{equation}\label{semi_algebraic_eq}
		[A,A^t]+S_A\circ_6 S_A=-\Big(\tr S_A^2-\frac{1}{2}(\tr JA)^2+2d\Big)I_6+D_1+D_1^t,
 	\end{equation}
    for some $D_1\in \mathfrak{gl}(\R^6)$ such that $[D_1,A]=dA$, where
    $$
      d=\frac{|[A,A^t]|^2+\langle S_A\circ_6 S_A, [A,A^t]\rangle}{2|A|^2}.
    $$
    In this case $Q_A=cI_7+\frac{1}{2}(D+D^t)$ for 
    \begin{equation}\label{eq: constant_c_semialgebraic}
      c=-\frac{1}{2}\Big(\tr S_A^2+\frac{1}{2}(\tr JA)^2+\frac{|[A,A^t]|^2}{|A|^2}+\frac{\langle S_A\circ_6 S_A,[A,A^t]\rangle}{|A|^2}\Big)
    \end{equation}
	\end{itemize}
\end{theorem}

\begin{proof}
\begin{itemize}
    \item[(i)]
	Suppose that $(\fg, A,\varphi)$ is an algebraic soliton i.e. $Q_A=cI+D$ for $c\in \mathbb{R}$ and $D\in\der(\fg,A)$. Then, 
	$$
	De_7=de_7 \qforq \text{some} \quad d\in \mathbb{R} \qandq [Q_A^\fh,A]=[D|_{\mathfrak{h}},A]=dA.
	$$ 
	Thus, by Proposition \ref{coflow_symbol} we get
	$$
	  [[A,A^t],A]+[S_A\circ_6 S_A,A]=2dA.
	$$
	Taking the inner product between $A$ and the above equation  we obtain 
	$$
	d=\frac{|[A,A^t]|^2+\langle S_A\circ_6 S_A,[A,A^t]\rangle}{2|A|^2}
	$$
	The converse follows by taking $D=Q_A-cI\in \der(\fg, A)$ and
	$$
	  c=q-d=-\frac{1}{2}\Big(\tr S_A^2+\frac{1}{2}(\tr JA)^2+\frac{|[A,A^t]|^2}{|A|^2}+\frac{\langle S_A\circ_6 S_A,[A,A^t]\rangle}{|A|^2}\Big).
	$$
	 \item[(ii)] Suppose that $(\fg,A,\varphi)$ is a semi algebraic soliton, i.e. $Q_A=cI_7+\frac{1}{2}(D+D^t)$ for some $c\in \mathbb{R}$ and $D\in \der(\fg, A)$. It implies the equations
 $$
 Q_A^\fh=cI_6+\frac{1}{2}(D_1+D_1^t) \qandq q=c+d
 $$
 where
	$$
	De_7=de_7 \qforq d\in \mathbb{R} \qandq [D_1,A]=dA \qwhereq D_1=D|_{\mathfrak{h}}.
	$$
 Since $\langle [D_1,A],A\rangle=\langle A,[D_1^t,A]\rangle$, by Proposition \ref{coflow_symbol} we obtain
 \eqref{semi_algebraic_eq}. The converse follows immediately, and the formulae for $c$ and $d$ are obtained as in (i).
\end{itemize}
\end{proof}

Using the condition \eqref{eq: algebraic_soliton-condition} we describe a class of algebraic solitons.

\begin{cor}\label{skew_sym_algebraic_soliton}
	If $A\in \mathfrak{sp}(\mathbb{R}^6)$ is skew-symmetric then $(\fg,A,\varphi)$ is an algebraic soliton.
\end{cor}
%From Lemma \ref{S_circ_S} we have
%$$
 % |S_A|^4=\frac{1}{4}|S_A\circ_6 S_A|^2+2|S_A^2|^2+\langle JS_A, S_A\rangle^2.
%$$
Using Lemma \ref{S_circ_S}, we can prove the absence of shrinking (semi-) algebraic solitons for the Laplacian coflow on almost Abelian Lie algebras.

\begin{prop}\label{prop_absences_solitons_steady}
	If $(\fg,A,\varphi)$ is a (semi-) algebraic soliton for the Laplacian coflow then it is expanding, and it is steady if it is torsion-free.
\end{prop}

\begin{proof}
Using the inequality \eqref{CSPP-inequality} in the equation \eqref{eq: constant_c_semialgebraic}, we have	
\begin{align*}
	2c%=&-\Big(\tr S_A^2+\frac{1}{4}(\tr JA)^2+\frac{|[A,A^t]|^2}{|A|^2}+\frac{\langle S_A\circ_6 S_A,[A,A^t]\rangle}{|A|^2}\Big)\\
     %\leq&-\Big(\tr S_A^2\frac{|A|^2}{|A|^2}+\frac{1}{2}(\tr JA)^2+\frac{|[A,A^t]|^2}{|A|^2}\Big)+\frac{|S_A\circ_6 S_A||[A,A^t]|}{|A|^2}\\
     %\leq&-\Big(|S_A|^2\frac{|S_A|^2}{|A|^2}+\frac{1}{4}(\tr JA)^2+\frac{|[A,A^t]|^2}{|A|^2}\Big)+\frac{|S_A\circ_6 S_A||[A,A^t]|}{|A|^2}\\
     %\leq&-\Big(\frac{|S_A\circ_6 S_A|^2}{4|A|^2}+2\frac{|S_A^2|^2}{|A|^2}+\frac{\langle JS_A, S_A\rangle^2}{|A|^2}+\frac{1}{4}(\tr JA)^2+\frac{|[A,A^t]|^2}{|A|^2}\Big)+\frac{|S_A\circ_6 S_A||[A,A^t]|}{|A|^2}\\
      %  =&-\Big(\frac{|S_A\circ_6 S_A|}{2|A|}-\frac{|[A,A^t]|}{|A|}\Big)^2-\Big(2\frac{|S_A^2|^2}{|A|^2}+\frac{\langle JS_A, S_A\rangle^2}{|A|^2}+\frac{1}{4}(\tr JA)^2\Big)\leq 0.
      \leq&-\Big(\tr S_A^2+\frac{1}{2}(\tr JA)^2+\frac{|[A,A^t]|^2}{|A|^2}\Big)+\frac{1}{|A|^2}\left(|S_A|^2|S_A|^2-2|S_A^2|^2-\langle JS_A,S_A\rangle^2+|[A,A^t]|^2\right)\\
      \leq&-\frac{1}{|A|^2}\Big(|S_A|^2(|A|^2-|S_A|^2)+\frac{1}{2}(\tr JA)^2|A|^2+2|S_A^2|^2+\langle JS_A,S_A\rangle^2\Big)\leq 0.
\end{align*}
If $c=0$ then 
$$
\tr JA=0 \qandq  S_A^2=0.
$$
 In particular $S_A=0$, and thus $A$ is skew-symmetric. And since $A\in \mathfrak{sp}(\R^6)$ it implies that $[J,A]=0$. Therefore, by equation \eqref{eq: torsion_tensor_A} we get that the full torsion tensor $T$ vanishes.
\end{proof}

\begin{remark}
We remark that the previous proposition was prove in \cite{Fino2018}*{Corollary 4.4} for the context of algebraic solitons and assuming that $A$ is normal.
\end{remark}
We conclude this section with an example of a semi-algebraic soliton which is not an algebraic one.

\begin{example}\label{ex: semi_algebraic_example}
Let $(\fg,A,\varphi)$ be an almost Abelian Lie algebra with $\gt$-structure $\varphi=\omega\wedge e^7+\rho^+$, where
\begin{equation*}
    \omega=e^{14}+e^{25}+e^{36} \qandq \rho^+=e^{123}-e^{156}+e^{246}-e^{345},
\end{equation*}
and the Lie bracket is given by the $3$-step nilpotent matrix 
	$$
	A=\left(\begin{array}{c|c}
	0 & B\\ \hline
	C & 0
	\end{array}\right)\in \mathfrak{sp}(\mathbb{R}^6), \quad B=\begin{pmatrix} 
	0 & 0 & 0\\
	0 & 0 & 1\\
	0 & 1 & 0
	\end{pmatrix}, \quad C=\begin{pmatrix}
	0 & \sqrt{2} & 0\\
	\sqrt{2} & 0 & 0\\
	0 & 0 & 0
	\end{pmatrix}.
	$$
	We have that the matrix
	$$
	D=\begin{pmatrix}
	D_1 & 0 \\
	0 & d
	\end{pmatrix}, \quad D_1=\left(\begin{array}{ccc|ccc}
	2 & 0 & 0 & & & \\
	0 & 2 & 0 & & & \\
	-\sqrt{2} & 0 & 4 & & & \\ \hline
	& & & 3 & 0 & \sqrt{2} \\
	& & & 0 & 3 & 0 \\
	& & & 0 & 0 & 1 
	\end{array}\right), \quad d=1,
	$$
	satisfies the relation $[D_1,A]=A$, it means that $D\in \der(\fg,A)$.  Now, for each term of \eqref{homogeneous_Q}, we obtain
	$$
	  [A,A^t]=\left(\begin{array}{ccc|ccc}
	  -2 & 0 & 0 & & & \\
	  0 & -1 & 0 & & & \\
	  0 & 0 & 1 & & & \\ \hline
	  & & & 2 & 0 & 0 \\
	  & & & 0 & 1 & 0\\
	  & & & 0 & 0 & -1
	  \end{array}\right), \quad S_A\circ_6 S_A=\left(\begin{array}{ccc|ccc}
	  1 & 0 & -\sqrt{2} & & & \\
	  0 & 0 & 0 & & & \\
	  -\sqrt{2} & 0 & 2 & & & \\ \hline
	  & & & -1 & 0 & \sqrt{2} \\
	  & & & 0 & 0 & 0\\
	  & & & \sqrt{2} & 0 & -2
	  \end{array}\right)
	$$
	and 
	$$\tr S_A^2= 3, \quad \tr JA=0, \quad |[A,A^t]|^2=12, \quad |A|^2=6 \qandq \langle S_A\circ_6 S_A,[A,A^t]\rangle=0.
	$$ %Notice that the matrix $(S_A\circ_6 S_A)_{ab}=S_A^{mn}S_A^{pq}\rho^+_{mpa}\rho^+_{nqb}$ was calculated using the $\SU(3)$-structure $\omega=e^{14}+e^{25}+e^{36}$ and $\rho^+=e^{123}-e^{156}+e^{246}-e^{345}$. 
	Since the matrices $A$ and $D$ satisfy the equation \eqref{semi_algebraic_eq}, we have that $(\fg,A,\varphi)$ is a semi-algebraic soliton with 
	$$
	  Q_A=-\frac{5}{2}I+\frac{1}{2}(D+D^t).
	$$
	Notice that $[D_1^t,A]\neq A$, so $D^t\notin \der(\fg,A)$ thus $(\fg,A,\varphi)$ is not an algebraic soliton. According to \eqref{eq: Auto_similar_sol_bracket}, the associated bracket flow solution is
 $$
A(t)=(1+5t)^{-1/2}e^{s(t)E}Ae^{-s(t)E}=(1+5t)^{-1/2}\left(\cos\frac{s(t)}{\sqrt{2}}A+\sin\frac{s(t)}{\sqrt{2}}A^\perp\right)
 $$
where 
$$
E=\frac12(D-D^t)=\frac{1}{\sqrt{2}}\left(\begin{array}{c|c|c}
  E_1& 0 & \\ \hline
	0 & E_1 & \\ \hline
 & & 0
\end{array}\right), \quad E_1=\left(\begin{array}{ccc}
    0 & 0 & 1  \\
	0 & 0 & 0 \\
	-1 & 0 & 0 
\end{array}\right)
$$
and
$$
A^\perp=\left(\begin{array}{c|c}
	0 & B'\\ \hline
	C' & 0
	\end{array}\right) \quad B'=\begin{pmatrix} 
	0 & -1 & 0\\
	-1 & 0 & 0\\
	0 & 0 & 0
	\end{pmatrix}, \quad C'=\begin{pmatrix}
	0 & 0 & 0\\
	0 & 0 & \sqrt{2}\\
	0 & \sqrt{2} & 0
	\end{pmatrix}.
$$
As in \cite{Lauret2017}*{Example 5.28}, we obtain that $A(t)/|A(t)|$ runs on a circle and $A(t)$ converges to zero rounding in a cone.
\end{example}

%%%%%%%%%%%%%%%%%%%%%%%%%%%%%%%%%
\appendix
\section{Contraction of $\gt$ and $\SU(3)$-identities}
 
\label{sect. Contraction}

Let $\varphi$ be a $\gt$-structure with Hodge dual $4$-form $\psi$ and induced $\SU(3)$-structure $(\omega,\rho^++i\rho^-)\in \Lambda^2(\R^6)^*\oplus\Lambda^3(\C^3)^*.$
From \cite{Karigiannis2007}*{\S A.3} and \cite{bedulli2007}*{\S 2.2}, we gather the following contraction identities for $\gt$ and $\SU(3)$-structures, respectively.

Contractions of $\varphi$ with $\varphi$: 
\begin{align}
    \varphi_{abc}\varphi^{abc}&=  42,\\
    \varphi_{abj}\varphi^{ab}_{\,\,\,k}
    &=  6g_{jk},\label{Eq.contraction.var.var.2 index}\\ 
    \varphi_{apq}\varphi^{a}_{\,\,jk}
    &=  g_{pj}g_{qk}-g_{pk}g_{qj}+\psi_{pqjk}.
    \label{Eq.contraction.var.var. 1 index}
\end{align}
Contractions of $\varphi$ with $\psi$:
\begin{align}
    \varphi_{ijk}\psi_{a}^{\,\,ijk}
    &=  0,\nonumber\\
    \varphi_{ijq}\psi^{ij}_{\,\,\,kl}&=  4\varphi_{qkl}, \label{Eq.contraction.var.psi.2 index}\\
    \varphi_{ipq}\psi^i_{\,\,jkl}
    &=  g_{pj}\varphi_{qkl}-g_{jq}\varphi_{pkl}+g_{pk}\varphi_{jql}\nonumber\\
    & -g_{kq}\varphi_{jpl}+g_{pl}\varphi_{jkq}-g_{lq}\varphi_{jkp}.
    \label{Eq.contraction.var..psi.1 index}
\end{align}
Contractions of $\psi$ with $\psi$:
\begin{align}
    \psi_{abcd}\psi^{ab}_{\,\,\,\,mn}
    &=  4g_{cm}g_{dn}-4g_{cn}g_{dm}+2\psi_{abmn},\label{Eq.contraction.psi.psi.2 index}\\
    \psi_{abcd}\psi_m^{\,\,\,bcd}
    &= 24g_{am},\label{Eq.contraction.psi.ps. 3 index}\\
    \psi_{abcd}\psi^{abcd}
    &= 168\nonumber
\end{align}
%The full torsion tensor is a $2$-tensor $T$ satisfying 
%\begin{align}
 %   \nabla_i\varphi_{jkl}
  %  &=  T_i^{\,\,\,m}\psi_{mjkl},\nonumber%\\
    %T_i^{\,\,j}
    %&= \frac{1}{24}\nabla_i\varphi_{lmn}\psi^{jlmn},
%\end{align}
%and
%\begin{equation}\label{eq: nabla.psi}
 %   \nabla_m\psi_{ijkl}=-(T_{mi}\varphi_{jkl}-T_{mj}\varphi_{ikl}-T_{mk}\varphi_{jil}-T_{ml}\varphi_{jki}).
%\end{equation}
%\section{Contraction identities of $\SU(3)$-structures}
%Given a $6$-dimensional vector space $V$ with inner product $\langle,\cdot,\cdot\rangle$ and  $\SU(3)$-structure, it means,  a pair $(\omega,\rho^+)\in \Lambda^2V^*\oplus\Lambda^3V^*$ satisfying 
%\begin{equation*}
 %   \omega\wedge \rho_+=0, \quad \omega^3=6\vol_V, \qandq \rho_+\wedge\rho_-=4\vol_V.
%\end{equation*}
%where $\rho^-=J^*\rho^+$ and $\omega(\cdot,\cdot)=\langle J\cdot,\cdot\rangle$. The following contractions identities hold \cite{bedulli2007}*{Section 2.2}
Contractions of $\omega$ with $\omega$ and $\rho^\pm$ with $\omega$:
\begin{align}\label{SU3_identities}
\omega_{ip}{\omega^p}_{j}=-\delta_{ij},\quad \rho^+_{iab}\omega^{ab}&=0, \quad  \rho^+_{ijp}{\omega^p}_{k}=\rho^-_{ijk}, \quad \rho^-_{ijp}{\omega^p}_{k}=-\rho^+_{ijk}.
\end{align}
Contractions of $\rho^\pm$ with $\rho^\pm$:
\begin{align}
\label{eq: rho_rho_2-index} 
\rho^+_{ipq}{\rho^-_j}^{pq}&=4\omega_{ij},\quad
\rho^+_{ipq}{\rho^+_j}^{pq}=4\delta_{ij}=\rho^-_{ipq}{\rho^-_j}^{pq},\\ \label{eq: rho-rho+}
\rho^-_{ijp}{\rho^+_{kl}}^p&=-\omega_{ik}\delta_{jl}+\omega_{jk}\delta_{il}+\omega_{il}\delta_{jk}-\omega_{jl}\delta_{ik},\\ \label{eq: rho+rho+}
\rho^+_{ijp}{\rho^+_{kl}}^p&=-\omega_{ik}\omega_{jl}+\omega_{il}\omega_{jk}+\delta_{ik}\delta_{jl}-\delta_{jk}\delta_{il}=\rho^-_{ijp}{\rho^-_{kl}}^p.\end{align}

\addcontentsline{toc}{section}{References}
\bibliography{Bibliografia-2020-07}

\Addresses
	
\end{document}